\documentclass[a4paper,12pt]{amsart}

\usepackage{fullpage}

\usepackage{amssymb,amsmath,amsthm}
\usepackage{centernot}
\usepackage{xspace}
\usepackage{enumerate}
\usepackage{multirow}
\usepackage{verbatim}

\usepackage[table]{xcolor}
\usepackage{hyperref}
\usepackage{url, hypcap}
\definecolor{darkblue}{RGB}{0,0,160}
\hypersetup{
colorlinks,%
citecolor=darkblue,%
filecolor=black,%
linkcolor=darkblue,%
urlcolor=darkblue
}

\usepackage[colorinlistoftodos]{todonotes}
\usepackage[utf8]{inputenc}

\numberwithin{equation}{section}

\usepackage{cleveref}
\crefname{conjecture}{conjecture}{conjectures}
\Crefname{conjecture}{Conjecture}{Conjectures}
\crefname{observation}{observation}{observations}
\Crefname{observation}{Observation}{Observations}
\crefname{hope}{hope}{hopes}
\Crefname{hope}{Hope}{Hopes}

\newtheorem{theorem}{Theorem}[section]
\newtheorem{lemma}[theorem]{Lemma}

\newtheorem{corollary}[theorem]{Corollary}
\newtheorem{proposition}[theorem]{Proposition}

\theoremstyle{definition}
\newtheorem{example}[theorem]{Example}
\newtheorem{remark}[theorem]{Remark}

\newtheorem{problem}[theorem]{Problem}

\newcommand{\darkblue}{\color{darkblue}} 
\newcommand{\defn}[1]{\emph{\darkblue #1}} 

\newcommand{\ring}[1]{\ensuremath{\mathbb{#1}}}

\newcommand\E{\mathbb{E}}
\newcommand\V{\mathbb{V}}

\newcommand\NN{\ring{N}}

\newcommand\QQ{\ring{Q}}
\newcommand\RR{\ring{R}}

\newcommand\ZZ{\ring{Z}}

\newcommand\cS{{\mathcal S}}

\newcommand\mmax{{m_{\operatorname{max}}}}

\newcommand\GF{\mathcal G}

\newcommand\dto{\xrightarrow{\ \mathcal{D}\ }}

\newcommand{\sage}{\textsc{Sage}\xspace}

\DeclareMathOperator{\Prob}{Prob} 

\def\Perm{{\mathfrak{S}}}
\def\stat{{\operatorname{st}}}
\def\inv{{\operatorname{inv}}}
\def\ides{{\operatorname{ides}}}
\def\des{{\operatorname{des}}}
\def\Inv{{\operatorname{Inv}}}
\def\Des{{\operatorname{Des}}}

\newcommand{\smallO}{o}
\newcommand{\bigO}{O}

\allowdisplaybreaks

\begin{document}

\title{Counting inversions and descents of \\ random elements in finite Coxeter groups}

\author[T.~Kahle]{Thomas Kahle}
\address[T.~Kahle]{Fakultät für Mathematik, OvGU Magdeburg, Magdeburg, Germany}
\urladdr{\url{http://www.thomas-kahle.de}}

\author[C.~Stump]{Christian Stump}
\address[C.~Stump]{Fakultät für Mathematik, Ruhr-Universit\"at Bochum, Germany}
\email{christian.stump@rub.de}


\begin{abstract}
  We investigate Mahonian and Eulerian probability distributions given by inversions and descents in general finite Coxeter groups.
  We provide uniform formulas for the means and variances in terms of Coxeter group data in both cases.
  We also provide uniform formulas for the double-Eulerian probability distribution of the sum of descents and inverse descents.
  We finally establish necessary and sufficient conditions for general sequences of Coxeter groups of increasing rank under which Mahonian and Eulerian probability distributions satisfy central and local limit theorems.
\end{abstract}

\maketitle

\setcounter{tocdepth}{1}
\tableofcontents

\section{Introduction}
\label{sec:introduction}

Properties of random permutations are important in many areas of applied mathematics, for example in statistical ranking where the collected data consists of permutations.
Instead of studying the actual permutations, applications often work with permutation statistics.
The most common include the numbers of cycles of various sizes, or the numbers of inversions and descents.
When permutations in the symmetric group are drawn uniformly at random, the asymptotics of the resulting random variables (as the size of the symmetric group tends to infinity) are well-studied.
Exact formulas for the moments and limit theorems for the corresponding distributions are known.
In this paper we extend the study of counting inversions and descents of random permutations to random elements of finite Coxeter groups.
We illustrate in detail how to compute means and variances, and follow the product formula approach by Bender~\cite{bender1973central} to give necessary and sufficient conditions on sequences of finite Coxeter groups of increasing rank such that the numbers of inversions and descents satisfy central and local limit theorems.
For permutations those are well-known phenomena.
We refer to~\cite{Bra2015,Bre19942,Pit1997} for these and further applications of Bender's approach.
Limit theorems for permutation statistics are a topic of continuing interest, we refer to~\cite{chatterjee2016central} for a recent consideration of the statistic given by the number of descents of a permutation plus the number of descents of its inverse.
We also provide uniform formulas for mean and variance of this statistic in general finite Coxeter groups.

\medskip

\Cref{sec:content} contains relevant notions for finite Coxeter groups and the associated random variables.
In \Cref{sec:Winversions,,sec:Wdescents,,sec:Wdesides}, we compute mean and variance of the \emph{$W$-Mahonian distribution} given by the number of inversions of a random Coxeter group element, the \emph{$W$-Eulerian distribution} given by the number of descents, and the \emph{$W$-double-Eulerian distribution} given by the number of descents plus the number of inverse descents.
In the final \Cref{sec:clts}, we exhibit necessary and sufficient conditions for central and local limit theorems to hold for the $W$-Mahonian and the $W$-Eulerian distributions.
These conditions turn out to only depend on the sizes of the dihedral parabolic subgroups in the sequence of Coxeter groups.
At the moment such necessary and sufficient conditions for limit theorems remain open for the $W$-double-Eulerian distribution of an arbitrary finite Coxeter group.

\medskip

This project began with an experimental investigation of the asymptotics of permutation statistics.
We present these investigations in \Cref{sec:findstat}.
In particular, we found the variances for the Mahonian, the Eulerian and the double-Eulerian distributions.
The first two are classical, while the latter was computed recently in~\cite{chatterjee2016central}.
Using the same procedure, we also found conjectured formulas for the other classical types~$B_n$ and $D_n$.
These are now \Cref{thm:Winversions,,thm:Wdescents,,thm:desides}.

In addition to means and variances of distributions of permutation statistics, one might as well try to guess formulas for higher moments and cumulants.
These computations can then suggest central limit theorems.
For Mahonian, Eulerian and double-Eulerian distributions in the symmetric group, the central limit theorems are known.
The first two have many different proofs, but the central limit theorem for the double-Eulerian distribution required some recent techniques~\cite{chatterjee2016central}.
Our experiments in the other classical types resulted in~\Cref{thm:WCLTinv,,thm:WCLTdes}.

\subsection*{Acknowledgements}

We thank Michael Drmota, Valentin Féray and Claudia Kirch for helpful discussions on conditions for central limit theorems, and Valentin Féray in particular for suggesting \Cref{prop:latticeCLT}.
We also thank Kyle Petersen, Christoph Thäle and Hugh Thomas for useful comments on a preliminary version of this paper.
We finally thank the anonymous referee for many helpful suggestions that improve the presentation of the paper.
In particular, we acknowledge the referee's suggestions that led to the uniform proof of~\Cref{thm:Wdescents} and to a complete revision of \Cref{sec:clts}.

\medskip

Thomas Kahle acknowledges support from the DFG (314838170, GRK 2297 MathCoRe).
Christian Stump was supported by the DFG grants STU 563/2 ``Coxeter-Catalan combinatorics'' and STU 563/4-1 ``Noncrossing phenomena in Algebra and Geometry''. 

\section{Probability distributions from Coxeter group statistics}
\label{sec:content}

A polynomial $f = \sum_i a_iz^i \in \NN[z]$ with $\NN = \{0,1,2,\ldots\}$ gives rise to a random variable $X_f$ on $\NN$ via
\[
  \Prob(X_f = k) = \frac{a_k}{\sum_i a_i} = [z^k]f / f(1).
\]
This is, the probability for $X_f$ to have value~$k$ is the coefficient of~$z^k$ in~$f$ divided by~$f(1)$.
A~\defn{permutation statistic} is, in its simplest form, a map
$
  \stat : \Perm_n \longrightarrow \NN,
$
where $\Perm_n$ is the group of permutations of $\{1,\ldots,n\}$.
Each such permutation statistic yields a random variable~$X_\stat$ on~$\NN$ when evaluated on permutation that is chosen uniformly at random.
These two concepts are linked via the generating function of a statistic
\[
\GF_{\stat}(z) = \sum_{\pi \in \Perm_n} z^{\stat(\pi)}
\]
since $\Prob(X_{\stat} = k) = \Prob(X_{\GF_{\stat}} = k)$.
In particular, the distribution of the random variable~$X_\stat$ only depends on the generating function of the statistic~$\stat$.

Two basic and important examples of permutation statistics are the number of
\defn{inversions} $\inv(\pi) = \#\Inv(\pi)$ (\href{http://www.findstat.org/St000018}{\tt findstat.org/St000018}) and of \defn{descents}
$\des(\pi) = \#\Des(\pi)$ of $\pi \in \Perm_n$ (\href{http://www.findstat.org/St000021}{\tt findstat.org/St000021}), where
\begin{align*}
  \Inv(\pi) &= \big\{ (i,j) \mid 1 \leq i<j \leq n, \pi(i) > \pi(j) \big\},\\
  \Des(\pi) &= \big\{ i \mid  1 \leq i < n, \pi(i) > \pi(i+1) \big\}.
\end{align*}
The \defn{Mahonian number} (\href{https://oeis.org/A000302}{\tt oeis.org/A000302}) is the number of permutations in $\Perm_n$ with~$k$ inversions and the \defn{Eulerian number} (\href{https://oeis.org/A008292}{\tt oeis.org/A008292}) is the number of permutations in $\Perm_n$ with~$k$ descents.
The Eulerian numbers have a long history.
Euler encountered them in the context of the evaluation of the sum of alternating powers $(1^{n}-2^{n}+3^{n}-\cdots)$.
The combinatorial definition that we use now became popular only during the 20th century.
See~\cite{Pet2015} for everything on Eulerian numbers.
The probability distributions for the random variables $X_\inv$ and $X_\des$ are respectively called \defn{Mahonian probability distribution} and the \defn{Eulerian probability distribution}.
Both are well studied, see~\cite{bender1973central} for a unified treatment.
Many extensions of these distributions are known.
Two examples are a central limit theorem for the Mahonian probability distribution on multiset permutations~\cite{canfield2011mahonian}, and a central limit theorem for Mahonian and Eulerian distribution on colored permutations~\cite{CM1012}.

\medskip

In this paper, we generalize and extend results about inversions and descents to
general finite Coxeter groups.
Let $(W,\cS)$ be a finite Coxeter group of rank~$n = |\cS|$.
The elements in $\cS$ are the \defn{simple reflections}.
Let $\Delta \subseteq \Phi^+ \subset \Phi = \Phi^+ \sqcup \Phi^-$ be a root system for $(W,\cS)$ with simple roots~$\Delta$ and positive roots~$\Phi^+$.  We refer
to~\cite[Part~1]{bjornerBrenti} for background on finite Coxeter groups.
Slightly abusing notation, we always think of a Coxeter group as coming with a fixed system of simple roots.
As usual, let $m({s,t})$ denote the order of the product $st \in W$ for two simple reflections $s \neq t$.
We set
\begin{equation}
\label{eq:mmax}
  \mmax = \mmax(W) = \max\big\{ m(s,t) \mid s,t \in S \big\}
\end{equation}
and observe that $2\mmax$ is the maximal size of a dihedral parabolic subgroup of~$W$.
All different products of the elements in $\cS$ are conjugate in~$W$ and thus have the same order~$h$.
If the Coxeter group~$W$ is irreducible, this number is called \defn{Coxeter number} of~$W$, and the eigenvalues of these elements are $\{ e^{2\pi i (d_k-1) / h}\}$ where $\{d_1,\ldots,d_n\}$ are the \defn{degrees} of~$W$.
The multiset of degrees of a reducible Coxeter group is the multiset union of the degree multisets of its irreducible components.

\medskip

For $w \in W$, one defines \defn{$W$-inversions} and \defn{$W$-descents} by
\[
  \Inv(w) = \big\{ \beta \in \Phi^+ \mid w(\beta) \in \Phi^- \big\}, \quad
  \Des(w) = \big\{ \beta \in \Delta \mid w(\beta) \in \Phi^- \big\},
\]
and we set $\inv(w) = \#\Inv(w)$ and $\des(w) = \#\Des(w)$.
These definitions specialize to the known definitions in the permutation group.
Positive roots in~$A_n = \Perm_{n+1}$ can be realized as
$\Phi^+ = \{ e_i - e_j \mid 1 \leq i < j \leq n+1\}$ and simple roots as
$\{ e_i - e_{i+1} \mid 1 \leq i \leq n\}$.  Therefore inversions and descents in
the one-line notation for $\Perm_{n+1}$ correspond to $A_n$-inversions and,
respectively, to~$A_n$-descents.  Consider for example the permutation
\[
  \pi = [2,5,1,3,6,4] = (12)(45)(34)(23)(56).
\]
In this case, we have
\begin{align*}
  \Inv(\pi) &= \{ 13, 23, 24, 26, 56 \} \leftrightarrow \{ e_1-e_3, e_2-e_3, e_2-e_4, e_2-e_4, e_2-e_6, e_5-e_6\},\\
  \Des(\pi) &= \{ \phantom{13,} \,2, \phantom{24, 26,}\ \ 5\phantom{6} \} \leftrightarrow \{ \phantom{e_1-e_3,}\, e_2-e_3, \phantom{e_2-e_4, e_2-e_4, e_2-e_6,}\, e_5-e_6\}.
\end{align*}

As above, the \defn{$W$-Mahonian numbers} and \defn{$W$-Eulerian numbers} are numbers of elements in~$W$ with exactly~$k$ $W$-inversions, and, respectively, $W$-descents.
The random variables $X_\inv$ and $X_{\des}$ are defined by the number of $W$-inversions and, respectively, the number of $W$-descents of a random element in~$W$.
Their distributions are given by the \defn{$W$-Mahonian distribution} and the \defn{$W$-Eulerian distribution} defined using their generating functions
\[
  \GF_{\inv}(W;z) = \sum_{w \in W}z^{\inv(w)}\quad\text{and}\quad \GF_\des(W;z) = \sum_{w \in W}z^{\des(w)}.
\]

\begin{remark}
\label{rem:statI}
  One could also study more general statistics interpolating between $W$-descents and $W$-inversions by defining $\stat_I(w) = \big\{ \beta \in I \mid w(\beta) \in \Phi^- \big\}$ where~$I$ is any subset of positive roots.
  At the end of \Cref{sec:Winversions}, we discuss how to analyze mean and variance of the distribution of any such statistic.
  However, the arguments for limit theorems depend on the concrete product structure of the generating functions, and do not apply to interpolating distributions in general.
\end{remark}

Given a product $W = W' \times W''$ of Coxeter groups, for both $\stat = \des$ and $\stat = \inv$ we have decompositions $\GF_{\stat} (W,z) = \GF_{\stat} (W',z)\cdot \GF_{\stat} (W'',z)$.
This corresponds to writing the random variable $\GF_{\stat} (W,z)$ as a sum of two independent random variables corresponding to $\GF_{\stat} (W',z)$ and $\GF_{\stat} (W'',z)$.
Therefore the computation of mean and variance for such variables on finite Coxeter groups reduces to the irreducible finite Coxeter groups.
We state this as the following lemma.
\begin{lemma}
  Let $W = W' \times W''$ be a product of two Coxeter groups $W'$ and~$W''$ and denote by $X_\stat$ either the number of inversions of a random element in~$W$ or the number of descents.
  Define $X'_\stat$ and $X''_\stat$ analogously.
  Then
  \begin{align*}
    \E(X_\stat) = \E(X'_\stat) + \E(X''_\stat), \qquad \V(X_\stat) = \V(X'_\stat) + \V(X''_\stat).
  \end{align*}
\end{lemma}

The main ingredients in the subsequent constructions from general finite Coxeter groups are the following properties of inversions and descents.
Following~\cite{stanley1989log}, a polynomial $f = a_nz^n + a_{n-1}z^{n-1} + \cdots +a_1z +a_0 \in \NN[z]$ is
\begin{itemize}
  \item \defn{unimodal} if $a_0 \leq \dots \leq a_{i-1} \leq a_i \geq a_{i+1} \geq \dots \geq a_n$ for some~$1 \leq i \leq n$, and
  \item \defn{log-concave} if $a_i^2 \geq a_{i-1}a_{i+1}$ for all $1 \leq i < n$.
\end{itemize}
If the sequence $a_0,\ldots,a_n$ has no internal zeroes, then log-concavity implies unimodality.
A stronger condition implying log-concavity is that~$f$ has only real nonpositive roots, that is,
$
  f = \prod_k(z+q_i)
$
with $q_i \in \RR_{\geq 0}$, see~\cite[Theorem~2]{stanley1989log}.

Let $[d]_z$ denote the \defn{$z$-integer}
$\frac{1-z^d}{1-z} = 1 +z + z^2 + \dots + z^{d-1}$ (often used as
\defn{$q$-integer}).  The following statement can be found for example in
\cite[Chapter~7]{bjornerBrenti}.
\begin{theorem}
\label{thm:Winvfactor}
Let~$W$ be a finite Coxeter group of rank~$n$ with degrees
$d_1,\ldots,d_n$.  The generating function for the number of inversions
satisfies
\begin{align}
  \GF_\inv(W;z) = \prod_{i=1}^n [d_i]_z. \label{eq:GFWinversions}
\end{align}
In particular, the sequence of coefficients of $\GF_\inv$ is log-concave and unimodal.
\end{theorem}

The next statement was proven in all irreducible types except type~$D$ in~\cite{Bre1994} while type~$D$ was only recently settled in~\cite{SV2015}.

\begin{theorem}
\label{thm:Wdesfactor}
  Let $(W,\cS)$ be a finite Coxeter group of rank~$n$.
  Then $\GF_\des$ has only real negative roots,
  \begin{align}
    \GF_\des(W;z) = \prod_{i=1}^n (z+q_i) \label{eq:GFWdescdents}
  \end{align}
  for some $q_1,\ldots,q_n \in \RR_{>0}$.  In particular, the sequence of coefficients of $\GF_\des$ is log-concave and unimodal.
\end{theorem}

\subsection{Inversions and descents in classical types}

The Coxeter group of type~$B_n$ can be realized as the
group of \defn{signed permutations}, that is antisymmetric bijections on
$\{\pm 1,\ldots,\pm n\}$.  In symbols,
\[
  B_n = \big\{ \pi : \{\pm 1,\ldots,\pm n\}\ \tilde\longrightarrow\ \{\pm 1,\ldots,\pm n\} \mid \pi(-i) = -\pi(i) \big\}.
\]
We represent signed permutations in their one-line notation $\pi = [\pi(1),\dots,\pi(n)]$ where $\pi(i) \in \{\pm 1,\dots,\pm n\}$ and $\{ |\pi(1)|,|\pi(2)|,\ldots,|\pi(n)| \} = \{1,\ldots,n\}$.
The Coxeter group of type~$D_n$ can be realized as the group of \defn{even signed permutations}, the subgroup of~$B_n$ of index~$2$ containing all signed permutations whose one-line notation contains an even number of negative
entries.
That is,
\[
  D_n = \big\{ \pi \in B_n \mid \pi(1)\cdot\pi(2)\cdot\ \cdots\ \cdot\pi(n) > 0 \big\}.
\]

Following~\cite[Prop.~8.1.1]{bjornerBrenti} in type~$B_n$ and~\cite[Prop.~8.2.1]{bjornerBrenti} in type~$D_n$, we set
\begin{align*}
  \Inv^+(\pi) &= \big\{1 \leq i < j \leq n \mid \pi(i) > \pi(j) \big\} \\
  \Inv^-(\pi) &= \big\{1 \leq i < j \leq n \mid -\pi(i) > \pi(j) \big\} \\
  \Inv^\circ(\pi) &= \big\{1 \leq i   \leq n \mid \pi(i) < 0 \big\}
\end{align*}
and obtain
\begin{equation}
\label{eq:Winversionsclassical}
  \Inv(\pi) = \begin{cases}
                \Inv^+(\pi) &\text{ for } \pi \in A_{n-1},\\
                \Inv^+(\pi) \cup \Inv^-(\pi) \cup \Inv^\circ(\pi) &\text{ for } \pi \in B_n,\\
                \Inv^+(\pi) \cup \Inv^-(\pi) &\text{ for } \pi \in D_n.
              \end{cases}
\end{equation}
Similarly, following~\cite[Prop.~8.1.2]{bjornerBrenti} in type~$B_n$ and~\cite[Prop.~8.2.2]{bjornerBrenti} in type~$D_n$, we set
\begin{equation}\label{eq:PiZeroDifferentTypes}
  \pi(0) = \begin{cases}
             0 &\text{ for } \pi \in A_{n-1},\\
             0 &\text{ for } \pi \in B_n,\\
             -\pi(2) &\text{ for } \pi \in D_n
           \end{cases}
\end{equation}
and define descents as
\begin{equation}
\label{eq:Wdescentsclassical}
  \Des(\pi) = \big\{0 \leq i < n \mid \pi(i) > \pi(i+1) \big\}.
\end{equation}

\section{The Mahonian distribution}
\label{sec:Winversions}

\begin{theorem}
\label{thm:Winversions}
  Let~$W$ be a finite Coxeter group.
  The $W$-Mahonian distribution~$X_\inv$ has mean and variance
  \[
    \E(X_\inv) = \frac{1}{2}\sum_{k=1}^n(d_k-1), \quad
    \V(X_\inv) = \frac{1}{12}\sum_{k=1}^n(d_k^2-1),
  \]
  where~$n$ is the rank of~$W$ and $d_1,\ldots,d_n$ are the degrees of~$W$.
\end{theorem}

The theorem can be written explicitly as follows.

\begin{corollary}
\label{cor:Winversions}
  In the situation of the previous theorem, the $W$-Mahonian distribution has means and variances
  \begin{align}
    \E(X_\inv) &= n(n+1)/4 &\qquad \V(X_\inv) &= (2n^3 + 9n^2 + 7n)/72 \tag{type $A_n$}\\
    \E(X_\inv) &= n^2/2 &\qquad \V(X_\inv) &= (4n^3 + 6n^2 - n)/36  \tag{type $B_n$}\\
    \E(X_\inv) &= n(n-1)/2 &\qquad \V(X_\inv) &= (4n^3 - 3n^2 - n)/36  \tag{type $D_n$} \\
    \E(X_\inv) &= 18 &\qquad \V(X_\inv) &= 29 \tag{type $E_6$} \\
    \E(X_\inv) &= 63/2 &\qquad \V(X_\inv) &= 287/4 \tag{type $E_7$} \\
    \E(X_\inv) &= 60 &\qquad \V(X_\inv) &= 650/3 \tag{type $E_8$} \\
    \E(X_\inv) &= 12 &\qquad \V(X_\inv) &= 61/3 \tag{type $F_4$} \\
    \E(X_\inv) &= 15/2 &\qquad \V(X_\inv) &= 137/12 \tag{type $H_3$} \\
    \E(X_\inv) &= 30 &\qquad \V(X_\inv) &= 361/3 \tag{type $H_4$} \\
    \E(X_\inv) &= m/2 &\qquad \V(X_\inv) &= (m^2+2)/12 \tag{type $I_2(m)$}
  \end{align}
\end{corollary}

We prove \Cref{thm:Winversions} using a well-known description of the generating function of the number of inversions in general finite Coxeter groups.
\Cref{cor:Winversions} follows from this description but we also provide an explicit proof in the classical types.

\begin{proposition}
\label{prop:Winversions}
Let $d_1,\ldots,d_n$ be \emph{any} sequence of positive integers and $X_f$ the random variable for the polynomial $f = \prod_{k=1}^n [d_k]_z$.
Then the mean and variance of~$X_f$ are
\[
\E(X_f) = \frac{1}{2}\sum(d_k-1), \quad \V(X_f) =
\frac{1}{12}\sum_{k=1}^n(d_k^2-1).
\]
\end{proposition}

\begin{proof}
For $d \ge 2$, let $X_d$ be the random variable for the polynomial~$[d]_z$.
That is, $X_d$ is distributed uniformly on the integers $\{0,\dots,d-1\}$.
A simple count yields that
\[
  X_f = X_{d_1} + \cdots + X_{d_n}
\]
for independent random variables~$X_{d_1},\ldots,X_{d_n}$.
Therefore, the mean and variance of $X_f$ are, respectively, the sums of the means and variances of the individual~$X_{d_k}$.
These are well-known to be $\E(X_d) = (d-1)/2$ and $\V(X_d)=\frac{1}{12}(d^2-1)$.
\end{proof}

\begin{proof}[Proof of \Cref{thm:Winversions}]
This is a direct application of \Cref{prop:Winversions}
given~\eqref{eq:GFWinversions}.
\end{proof}

For the proof of \Cref{cor:Winversions} it is now sufficient to look up the
degrees of the irreducible finite Coxeter groups given by
  \begin{align}
    &2,3,\dots,n+1                 \tag{type $A_n$}\\
    &2,4,\dots,2n                  \tag{type $B_n$}\\
    &2,4,\dots,2n-2,n              \tag{type $D_n$} \\
    &2, 5, 6, 8, 9, 12             \tag{type $E_6$} \\
    &2, 6, 8, 10, 12, 14, 18       \tag{type $E_7$} \\
    &2, 8, 12, 14, 18, 20, 24, 30  \tag{type $E_8$} \\
    &2, 6, 8, 12                   \tag{type $F_4$} \\
    &2, 6, 10                      \tag{type $H_3$} \\
    &2, 12, 20, 30                 \tag{type $H_4$} \\
    &2,m                           \tag{type $I_2(m)$}
  \end{align}

We also discuss an instructive direct proof, using combinatorial interpretations of inversions in the classical types.
We then describe how to use such sum decompositions to analyze the variance of any statistic $\stat_I$ for $I \subseteq \Phi^+$ as in \Cref{rem:statI}.

\medskip

To this end, define indicator random variables corresponding to the three sets in~\eqref{eq:Winversionsclassical}.
\begin{align*}
  Y^+_{ij} & =
  \begin{cases}
    1 & \text{if } \pi(i) > \pi(j)\\
    0 & \text{otherwise }
  \end{cases}\\
  Y^-_{ij} & =
  \begin{cases}
    1 & \text{if } -\pi(i) > \pi(j)\\
    0 & \text{otherwise }
  \end{cases}\\
  Y^\circ_{i} & =
  \begin{cases}
    1 & \text{if } \pi(i) < 0\\
    0 & \text{otherwise }
  \end{cases}
\end{align*}
These random variables can be interpreted as indicating how $\pi$ acts on the
positive roots if one identifies
\begin{align}
\label{eq:variablesroots}
  Y^+_{ij}  \leftrightarrow e_i - e_j, \qquad
  Y^-_{ij} \leftrightarrow e_i + e_j, \qquad
  Y^\circ_{i} \leftrightarrow e_i
\end{align}

With these definitions and \eqref{eq:Winversionsclassical} we have
\begin{align}
  X_{\inv} & = \sum_{i<j} Y^+_{ij} \tag{type $A_{n-1}$}\\
  X_{\inv} & = \sum_{i<j} Y^+_{ij} + \sum_{i<j} Y^-_{ij} + \sum_{i} Y^\circ_{i}  \tag{type $B_n$}\\
  X_{\inv} & = \sum_{i<j} Y^+_{ij} + \sum_{i<j} Y^-_{ij} \tag{type $D_n$}.
\end{align}

For the alternative proof of \Cref{cor:Winversions}, using $\V(X) = \E(X^{2}) - \E(X)^{2}$, one needs to control the covariances among the random variables.
The mean of $X_{\inv}$ is easily confirmed as a warm-up to the following computation recalculating $\V(X_\inv)$ in type~$B_n$:
\begin{align}
  \E(X_{\inv}^{2} ) & = \E\big(\sum_{i<j} Y^+_{ij} + \sum_{i<j} Y^-_{ij} + \sum_{i}
          Y^\circ_{i}\big)^{2} \nonumber\\
        & = \binom{n}{2}\frac{1}{2} +
          \binom{n}{2}\binom{n-2}{2}\frac{1}{4} +
          2 \binom{n}{3}\frac{1}{6} +
          4\binom{n}{3}\frac{1}{3} \tag {$Y^+$ with $Y^+$}\\
        & + \binom{n}{2}\frac{1}{2} +
          \binom{n}{2}\binom{n-2}{2}\frac{1}{4} +
          2 \binom{n}{3}\frac{1}{3} +
          4\binom{n}{3}\frac{1}{3} \tag {$Y^-$ with $Y^-$}\\
        & + n\frac{1}{2} +
          2 \binom{n}{2}\frac{1}{4} \tag {$Y^\circ$ with $Y^\circ$}\\
        & + 2 \Bigg[
          \binom{n}{2}\frac{1}{4} +
          \binom{n}{2}\binom{n-2}{2}\frac{1}{4} +
          \binom{n}{3}\frac{1}{3} +
          \binom{n}{3}\frac{1}{6} +
          2\binom{n}{3}\frac{1}{3} +
          2\binom{n}{3}\frac{1}{6} \tag {$Y^+$ with $Y^-$}\\
        & \qquad +
          3\binom{n}{3}\frac{1}{4} +
          \binom{n}{2}\frac{1}{8} +
          \binom{n}{2}\frac{3}{8} \tag {$Y^+$ with $Y^\circ$}\\
        & \qquad +
          3\binom{n}{3}\frac{1}{4} +
          \binom{n}{2}\frac{3}{8} +
          \binom{n}{2}\frac{3}{8} \tag {$Y^-$ with $Y^\circ$} \Bigg]\\
        & = \frac{1}{4}n^{4} + \frac{1}{36}(4n^{3} + 6n^{2} - n)\notag
\end{align}
The formula is written so that each summand is given by the product of the ``number of occurrences of a pattern'' times the ``probability of this pattern''.
This is the number of indices $ij,kl$ (or $ij,k$) of a given pattern times the probability that $Y_{ij}Y_{kl} = 1$ (or, respectively, $Y_{ij}Y_k = 1$).
Working out all the summands is simple and instructive.
As an example, the two summands in ``$Y^\circ$ with $Y^\circ$'' are given by
\[
  \E\big((\sum_iY_i^\circ)^2\big) = \sum_i \E(Y_i^\circ) + 2\sum_{i<j} \E(Y_i^\circ Y_j^\circ) = n\frac{1}{2} + 2\binom{n}{2}\frac{1}{4}
\]
because the $Y_i^\circ$ are independent among each other and $\E(Y_i^\circ)=1/2$.
After subtracting $\E(X)^{2} = \frac{1}{4}{n^{4}}$ from the result above we find
\Cref{cor:Winversions} in type~$B_n$.
The variance formulas for types~$A_{n-1}$ and $D_{n}$ can be deduced from above, omitting all terms that contain $Y^-$ or $Y^\circ$ in type~$A_{n-1}$ and those that contain $Y^\circ$ in type~$D_{n}$.

\medskip

The same argument can also be used to analyze the distribution~$X_{\stat_{I}}$ of any statistic $\stat_I = w \mapsto \#\big\{ \beta \in I \mid w(\beta) \in \Phi^- \big\}$ where~$I$ is any subset of positive roots as in \Cref{rem:statI}.
First, there is a uniform argument to compute the mean.
\begin{proposition}
\label{prop:interpolatingstatistic}
  Let~$W$ be a finite Coxeter group and let~$I \subseteq \Phi^+$ be a subset of positive roots.
  Then
  \[
    \E(X_{\stat_I}) = \tfrac{1}{2}|I|\ .
  \]
\end{proposition}
\begin{proof}
  Let~$w_\circ \in W$ be the unique element with $\Inv(w_\circ) = \Phi^+$.
  Then
  \[
    \Inv(w) \cup \Inv(w_\circ w) = \Phi^+, \quad \Inv(w) \cap \Inv(w_\circ w) = \emptyset.
  \]
  Since $\stat_I(w) = |\Inv(w) \cap I|$, we obtain that $\stat_I(w) + \stat_I(w_\circ w) = |I|$ and the statement follows because $w \mapsto w_\circ w$ is a bijection (indeed an involution) on~$W$.
\end{proof}

To obtain the variance of $\stat_I$ as well, one proceeds as in the direct proof of \Cref{cor:Winversions}, this time using only the variables $Y^+_{ij}, Y^-_{ij}, Y^\circ_{i}$ corresponding to positive roots in~$I$.  The matching is as in~\eqref{eq:variablesroots} and
\[
  X_{\stat_I} = \sum_{ \beta \in I} X_\beta
\]
where $X_\beta$ is the random variable corresponding to the positive root~$\beta \in I$.

\section{The Eulerian distribution}
\label{sec:Wdescents}

\begin{theorem}
\label{thm:Wdescents}
  Let $(W,\cS)$ be an irreducible finite Coxeter group of rank at least two and let $m = \mmax$ denote half the size of a dihedral parabolic subgroup of~$W$ as in~\eqref{eq:mmax}.
  The $W$-Eulerian distribution~$X_\des$ has mean and variance
  \[
    \E(X_\des) = n / 2, \quad
    \V(X_\des) = (n-2)/12 + 1/m,
  \]
  where~$n$ is the rank of~$W$.
\end{theorem}

The theorem can be written explicitly as follows.

\begin{corollary}
\label{cor:Wdescents}
  The variances of the $W$-Eulerian distributions in \Cref{thm:Wdescents} satisfy
  \begin{align}
    \V(X_\des) &=(n+2)/12 \tag{type $A_n$}\\
    \V(X_\des) &= (n+1)/12 \tag{type $B_n$} \\
    \V(X_\des) &=(n+2)/12 \tag{type $D_n$} \\
    \V(X_\des) &=(n+2)/12 \tag{type $E_n$} \\
    \V(X_\des) &= 5/12 \tag{type $F_4$} \\
    \V(X_\des) &= 17/60 \tag{type $H_3$} \\
    \V(X_\des) &= 11/30 \tag{type $H_4$} \\
    \V(X_\des) &= 1/m \tag{type $I_2(m)$}
  \end{align}
\end{corollary}

\begin{remark}
  The groups of types~$A_{n-1}$ and~$B_n$ are also wreath products $\mathcal{C}_r \wr\Perm_{n}$ where $\mathcal{C}_r$ is the cyclic group on~$r$ letters.
  In~\cite{CM1012}, Chow and Mansour consider the distributions of various statistics on these groups, including the number of descents.
  For this statistic, Steingr\'imson's formula for the generating functions yields mean and variance.
  Then, using a theorem of Aissen, Schoenberg and Whitney, Chow and Mansour find that the coefficient sequences of the generating functions are log-concave and from this central and local limit theorems can be derived.
\end{remark}

\begin{remark}
  \Cref{thm:Wdescents} can be used to obtain information about the (negatives of the) roots of $\GF_\des(W;z) = \prod_i (z+q_i)$, since one may compute, as done in \cite[Theorem~2]{bender1973central},
  \[
    \E(X_\des) = \sum_{i=1}^n\frac{1}{1+q_i}, \qquad \V(X_\des) = \sum_{i=1}^n\frac{q_i}{(1+q_i)^2}.
  \]
  Observe that the palindromicity $\GF_\des(W;z) = z^n\cdot\GF_\des(W;z^{-1})$ implies that the equation for the mean is trivially satisfied because the roots come in inverse pairs~$q$ and~$q^{-1}$.
  On the other hand, we are not aware of any previously known property of the roots which implies the equation for the variance.
\end{remark}

The proof of \Cref{thm:Wdescents} can be deduced from the following lemma used to control the covariances among the individual descents contributing to~$X_\des$.
The lemma can be found for example in~\cite[Corollary~2.4.5(ii)]{bjornerBrenti}.

\begin{lemma}
\label{lem:cosetdocomposition}
  Let $(W,\cS)$ be a finite Coxeter group.
  For $J \subseteq \cS$ denote by $W_J$ the subgroup of~$W$ generated by~$J$ and set $\mathcal{D}_J = \{ w \in W \mid J \subseteq \Des(w) \}$.
  Then $\mathcal{D}_J$ is a complete list of coset representatives of $W / W_J = \{ wW_J \mid w \in W\}$.
  Moreover, $|W| = |W_J| \cdot |\mathcal{D}_J|$.
\end{lemma}

\begin{proof}[Proof of~\Cref{thm:Wdescents}]
  The proof for the mean follows from its linearity together with \Cref{lem:cosetdocomposition} as follows.
  Given any $s \in \cS$, we have $|W_{\{s\}}| = 2$ and thus,
  \[
    \E(X_\des) = \sum_{s \in \cS} \frac{|\mathcal{D}_{\{s\}}|}{|W|} = n/2.
  \]
  Here, we used that $\mathcal{D}_{\{s\}}$ contains exactly the elements in~$W$ having~$s$ as a descent.
  We next compute the variance as
  \begin{align*}
    \V(X_\des) = \E(X_\des^2) - \E(X_\des)^2
    &= \sum_{s,t \in \cS} \frac{|\mathcal{D}_{\{s,t\}}|}{|W|} - \frac{n^2}{4} \\
    &= \frac{n}{2} + \sum_{s \neq t} \frac{|\mathcal{D}_{\{s,t\}}|}{|W|}  - \frac{n^2}{4} \\
    &= \frac{n}{2} + \frac{(n-1)(n-2)}{4} + \frac{n-2}{3} + \frac{1}{m} - \frac{n^2}{4} \\
    &= \frac{n-2}{12} + \frac{1}{m}.
  \end{align*}
  Here, the first equation is the definition, the second equation is the linearity of the mean, the third equation uses that the~$n$ summands with $s=t$ contribute $1/2$ each.
  The fourth equation is obtained as follows.
  According to \Cref{lem:cosetdocomposition}, each pair $s \neq t$ contributes $1/|W_{\{s,t\}}|$, and $|W_{\{s,t\}}| = 2m({s,t})$.
  The Coxeter diagram of an irreducible Coxeter group is a tree having at most one label $m > 3$.
  Therefore, there are
  $2\big(\binom{n}{2}-(n-1)\big) = (n-1)(n-2)$ summands $s \neq t$ with $m(s,t) = 2$, each contributing $1/4$,
  there are $2(n-2)$ summands $s \neq t$ with $m(s,t)=3$, each contributing $1/6$, and
  there are two summands $s \neq t$ with $m(s,t)=m$, each contributing $\tfrac{1}{2m}$.
\end{proof}

As in the previous section, we also discuss an alternative direct proof using the combinatorial interpretations of descents in~\eqref{eq:Wdescentsclassical}.
We start with defining the indicator random variables
\begin{equation}
\label{eq:Xi}
  Y^{(i)} =
  \begin{cases}
  1 & \text{$\pi(i) > \pi(i+1)$}\\
  0 & \text{otherwise}.
  \end{cases}
\end{equation}
The definition of $Y^{(i)}$ is different in each type because of \eqref{eq:PiZeroDifferentTypes}.
In every case, the number of descents of a random element $\pi \in W$ is the sum of such random variables and mean and variance can be computed from this sum since~\eqref{eq:Wdescentsclassical} implies that
\begin{align}
\label{desvars}
  X_\des = \sum_{i=0}^{n-1} Y^{(i)}
\end{align}
in types~$B_n$ and~$D_n$, while the sum is from~$1$ to~$n$ in type~$A_n$.
The $A_n$-case is well-known.
\begin{proposition}
\label{prop:Adescents}
  The mean and variance of the Eulerian distribution on $A_n$ are
  \[
    \E(X_\des) = \frac{n}{2},\qquad \V(X_\des) = \frac{n+2}{12}
  \]
\end{proposition}
\begin{proof}
The mean is clear from linearity and $\E(Y^{(i)}) = 1/2$.
To compute $\E(X_\des^2) = \sum_{i,j}\E(Y^{(i)}Y^{(j)}) $ we distinguish three
types of summands:
\begin{itemize}
  \item The~$n$ summands with $i=j$ give $\E(Y^{(i)}Y^{(j)}) = 1/2$.

  \item The $2(n-1)$ summands with $|i-j|=1$ give $\E(Y^{(i)}Y^{(j)})=1/6$,
  since $\pi(a) > \pi(a+1) > \pi(a+2)$ for $1 \leq a < n-1$ occurs exactly once among the six equally likely possibilities.

  \item For the summands with $|i-j|>1$ we have $\E(Y^{(i)}Y^{(j)})=\E(Y^{(i)})\E(Y^{(j)})=1/4$.
\end{itemize}

We thus find
\begin{align*}
  \V(X_\des) &= \E(X_\des^2) - \E(X_\des)^2 \\
       &= \frac{n}{2} + \frac{2(n-1)}{6} + \frac{n^2 - n - 2(n-1)}{4} - \frac{n^2}{4} \\
       &= \frac{n+2}{12}. \qedhere
\end{align*}
\end{proof}

\begin{proposition}
\label{prop:Bdescents}
The mean and variance of the $B_n$-Eulerian distribution are
\[
\E(X_\des) = \frac{n}{2},\qquad \V(X_\des) = \frac{n+1}{12}
\]
\end{proposition}

\begin{proof}
Again, $\E(X_\des) = n/2$ is clear from linearity of $\E$.
To compute $\E(X_\des^2)$ we split the sum over pairs $i,j\in\{0,\dots,n-1\} $ into four types of summands.

\begin{itemize}
  \item The~$n$ summands with $i=j$ give $\E(Y^{(i)}Y^{(j)}) = 1/2$.
  \item The $2(n-2)$ summands with $|i-j|=1$ and $i,j>0$ give $\E(Y^{(i)}Y^{(j)})=1/6$ for the
same reason as in \Cref{prop:Adescents}.
  \item The~$2$ summands with $\{i,j\} = \{0,1\}$ give $\E(Y^{(i)}Y^{(j)}) = 1/8$.
        This is because $0 > \pi(1) > \pi(2)$ occurs in exactly one of eight equally likely possibilities $\pi(1),\pi(2) \lessgtr 0$ and $\pi(1) > \pi(2)$.
  \item Finally, the $n^2-n-2(n-2)-2$ summands with $|i-j|>1$ give $\E(Y^{(i)}Y^{(j)}) =\E(Y^{(i)})\E(Y^{(j)})=1/4$.
\end{itemize}

We thus find
\begin{align*}
  \V(X_\des) &= \E(X_\des^2) - \E(X_\des)^2 \\
       &= \frac{n}{2} + \frac{2(n-2)}{6} + \frac{2}{8} + \frac{n^2-n-2(n-2)-2}{4} - \frac{n^2}{4} \\
       &= \frac{n+1}{12}. \qedhere
\end{align*}
\end{proof}

\begin{proposition}
\label{prop:Ddescents}
The mean and variance of the $D_n$-Eulerian distribution are
\[
\E(X_\des) = \frac{n}{2}, \qquad \V(X_\des) = \frac{n+2}{12}.
\]
\end{proposition}

\begin{proof}
  By linearity of $\E$ again $\E(X_\des) = n/2$.  To compute $\E(X_\des^2)$ we
  here consider five types of pairs $i,j\in\{0,\dots,n-1\}$.

  \begin{itemize}
    \item The~$n$ summands with $i=j$ give $\E(Y^{(i)}Y^{(j)}) = 1/2$.

    \item The $2(n-2)$ summands with $|i-j|=1$ and $i,j>0$ give
    $\E(Y^{(i)}Y^{(j)})=1/6$ for the same reason as in \Cref{prop:Adescents}.

    \item The~$2$ summands with $\{i,j\} = \{0,1\}$ yield
    $\E(Y^{(i)}Y^{(j)}) = 1/4$ since one quarter of the elements of $D_n$
    satisfies $-\pi(2) > \pi(1) > \pi(2)$.

    \item The~$2$ summands with $\{i,j\} = \{0,2\}$ yield $\E(Y^{(i)}Y^{(j)}) = 1/6$.
          This is because one asks how often $-\pi(3) > -\pi(2) > \pi(1)$.

    \item Finally, in all other summands $\E(Y^{(i)}Y^{(j)})=\E(Y^{(i)})\E(Y^{(j)}) = 1/4$.
  \end{itemize}
In total we have
  \begin{align*}
    \V(X_\des) &= \E(X_\des^2) - \E(X_\des)^2 \\
        &= \frac{n}{2} + \frac{2(n-2)}{6} + \frac{2}{4} + \frac{2}{6} + \frac{n^2-n-2(n-2)-4}{4} - \frac{n^2}{4} \\
        & = \frac{n+2}{12}. \qedhere
  \end{align*}
\end{proof}

\begin{proof}[Proof of \Cref{cor:Wdescents}]
  The classical types are dealt with in \Cref{prop:Adescents,,prop:Bdescents,,prop:Ddescents}.
  The computation in the dihedral types~$I_2(m)$ is obvious, and the remaining were computed using
  \sage~\cite{sagemath}.
\end{proof}

\section{The double-Eulerian distribution}
\label{sec:Wdesides}

An \defn{inverse descent} (also known as \defn{recoil} or \defn{ligne of route}) of a permutation~$\pi$ is a descent of $\pi^{-1}$,
\[
  \ides(\pi) = \des(\pi^{-1}).
\]
Permutations with~$k$ descents and~$\ell$ inverse descents have been studied in various contexts;
we refer to the unpublished manuscript by Foata and Han~\cite{FH2011} for a detailed combinatorial treatment of this bi-statistic.
To emphasize its bivariate nature, we refer to the numbers of permutations with~$k$ descents and~$\ell$ inverse descents as the \defn{bi-Eulerian numbers} and to the numbers of permutations such that $\des(\pi) + \ides(\pi)$ equals~$k$ as the \defn{double-Eulerian numbers} (\href{https://oeis.org/A298248}{\tt oeis.org/A298248}).
Several papers use the term double-Eulerian numbers already for the bivariate version.
Others, such as~\cite{Pet2013}, refer to the bi-statistic as the \emph{two-sided Eulerian numbers}.
We have chosen the present terms in order to clarify the distinction between the bivariate statistic $(\des(\pi),\ides(\pi))$ and the univariate statistic $\des(\pi) + \ides(\pi)$ (\href{http://www.findstat.org/St000824}{\tt findstat.org/St000824}).
We thus call the probability distributions for the random variables $X_{\des+\ides}$ \defn{double-Eulerian probability distribution}.

In type $A_{n}$, Chatterjee and Diaconis~\cite{chatterjee2016central} computed the mean and variance of the double-Eulerian distribution as 
\[
  \E(X_{\des+\ides}) = n, \qquad \V(X_{\des+\ides}) = \frac{n+8}{6} - \frac{1}{n+1}.
\]
We generalize this result uniformly to all finite Coxeter groups.

\begin{theorem}
\label{thm:desides}
  Let~$W$ be an irreducible finite Coxeter group of rank~$n$ and Coxeter number~$h$.
  Then
  \begin{align}
    \E(X_{\des+\ides}) = n, \qquad \V(X_{\des+\ides}) &= 2 \V(X_\des) + n/h.
  \end{align}
\end{theorem}
The theorem can be written explicitly as follows.
\begin{corollary}
\label{cor:desides}
  In the situation of the previous theorem, the $W$-double-Eulerian distribution has variances
  \begin{align}
    \V(X_{\des+\ides}) &= \frac{n+2}{6} + \frac{n}{n+1} \tag{type $A_n$} \\
    \V(X_{\des+\ides}) &= \frac{n+4}{6} \tag{type $B_n$} \\
    \V(X_{\des+\ides}) &= \frac{n+2}{6} + \frac{n}{2n-2} \tag{type $D_n$} \\
    \V(X_{\des+\ides}) &= 11/6 \tag{type $E_6$} \\
    \V(X_{\des+\ides}) &= 17/9 \tag{type $E_7$} \\
    \V(X_{\des+\ides}) &= 29/15 \tag{type $E_8$} \\
    \V(X_{\des+\ides}) &= 7/6 \tag{type $F_4$} \\
    \V(X_{\des+\ides}) &= 13/15 \tag{type $H_3$} \\
    \V(X_{\des+\ides}) &= 13/15 \tag{type $H_4$} \\
    \V(X_{\des+\ides}) &= 4/m \tag{type $I_2(m)$}
  \end{align}
\end{corollary}

In this case of descents plus inverse descents, we do not have a uniform argument for the variances.
Before providing a case-by-case analysis of the situation, we present a corollary concerning double cosets in finite Coxeter groups.
The following lemma can for example be found in \cite[Proposition~2.7(b)]{BKPST2018}.

\begin{lemma}
\label{lem:doublecosetdocomposition}
  Let $(W,\cS)$ be a finite Coxeter group.
  For $I,J \subseteq \cS$, set $_I\mathcal{D}_J = \{ w \in W \mid J \subseteq \Des(w) \text{ and } I \subseteq \Des(w^{-1}) \}$.
  Then $_I\mathcal{D}_J$ is a complete list of double coset representatives of $W_I \backslash W / W_J = \{ W_IwW_J \mid w \in W\}$.
\end{lemma}

Observe that double cosets are, in general, not all of the same cardinality.
In particular, the previous lemma does not provide a uniform counting formula for the set $_I\mathcal{D}_J$.
Given \Cref{thm:desides}, one may now deduce a uniform sum count of all cardinalities of double cosets of the form $W_I \backslash W / W_J$ with $|I| = |J| = 1$.

\begin{corollary}
  Let $(W,\cS)$ be a finite Coxeter group of rank~$n$ with Coxeter number~$h$.
  Then
  \[
    \sum_{s,t \in \cS} \big| W_{\{s\}} \backslash W / W_{\{ t \}} \big| = \frac{n}{4h}(nh+2).
  \]
\end{corollary}

\begin{proof}
  \Cref{lem:doublecosetdocomposition} shows that $\big| W_{\{s\}} \backslash W / W_{\{ t \}} \big|$ equals the number of elements in~$W$ having~$t$ as a descent and~$s$ as an inverse descent.
  The linearity of the mean thus implies that
  \begin{align*}
    \V(X_{\des+\ides}) &= \E(X_{\des+\ides}^2) - \E(X_{\des+\ides})^2 \\
                       &= 2\E(X_{\des}^2) + 2\sum_{s,t \in \cS} \big| W_{\{s\}} \backslash W / W_{\{ t \}} \big| - \big( 2\E(X_{\des})^2 + n^2/2 \big) \\
                       &= 2\V(X_\des) + 2\sum_{s,t \in \cS} \big| W_{\{s\}} \backslash W / W_{\{ t \}} \big| - n^2/2.
  \end{align*}
  The desired conclusion is therefore equivalent to the conclusion in \Cref{thm:desides}.
\end{proof}

We turn to the proof of \Cref{thm:desides}, which we divide into three propositions, one for each type.
In analogy to the random variables $Y^{(i)}$ from~\eqref{desvars}, define
\[
  \tilde Y^{(j)} =
  \begin{cases}
    1 & \text{$\pi^{-1}(j) > \pi^{-1}(j+1)$},\\
    0 & \text{otherwise}.
  \end{cases}
\]
Using the two sets of random variables we write
\begin{align}
\label{eq:desidesvar}
  X_{\des+\ides} = \sum_{i=1}^n \big(Y^{(i)} + \tilde Y^{(i)} \big).
\end{align}

\begin{remark}
\label{r:invdes}
  The locations of inverse descents of $\pi$ can be read off the one-line notation.
  In type $A$, $j$ is an inverse descent if the location of $j+1$ is to the left of the location of~$j$.
  In types~$B$ and~$D$ the signs also play a role.
  Specifically, $\pi^{-1}(j) > \pi^{-1}(j+1)$ if one of the following four orderings occurs
  \[
    j+1 \text{ left of }j
    \quad\text{or}\quad
    -(j+1)\text{ left of }j
    \quad\text{or}\quad
    j\text{ left of }-\!(j+1)
    \quad\text{or}\quad
    -j\text{ left of }-\!(j+1).
  \]
\end{remark}

\begin{proposition}
\label{prop:Adesides}
  The mean and variance of the distribution $X_{\des+\ides}$ on $A_n$ are
  \[
    \E(X_{\des+\ides}) = n,\qquad \V(X_{\des+\ides}) = 2\V(X_\des) + {n}/{(n+1)}.
  \]
\end{proposition}

\begin{proof}
  The computation for the mean is obvious.
  For the variance, we first record that~\eqref{eq:desidesvar} implies that
  \begin{align*}
    \V(X_{\des+\ides}) &= \E(X_{\des+\ides}^2) - \E(X_{\des+\ides})^2 \\
                       &= 2\E(X_{\des}^2) + 2\sum_{i,j=1}^n \E(Y^{(i)}\tilde Y^{(j)}) - n^2 \\
                       &= 2\V(X_\des) + 2\sum_{i,j=1}^n \E(Y^{(i)}\tilde Y^{(j)}) - n^2/2
  \end{align*}
  where we used that the distributions $X_{\des}$ and $X_\ides$ coincide and that $n = \E(X_{\des+\ides}) = 2\E(X_\des)$.
  We thus aim to show that
  \[
    2\sum_{i,j=1}^n \E(Y^{(i)}\tilde Y^{(j)}) - \frac{n^2}{2} = \frac{n}{n+1}.
  \]
  For fixed $1 \leq i,j \leq n$, by \Cref{r:invdes}, $Y^{(i)}\tilde Y^{(j)} = 1$ if and only if $\pi(i) > \pi(i+1)$ and $j,j+1$ are out of order in the one-line notation of~$\pi$.
  We claim the following expression for the mean:
  \begin{multline*}
    \E(Y^{(i)}\tilde Y^{(j)}) = \frac{1}{(n+1)!}
      \Big[ \tfrac{1}{4}(n-1)(n-2)(n-1)! + (n-1)! \\+ (n-2)!\big((i-1)(j-1) + (i-1)(n-j) + (n-i)(j-1) + (n-i)(n-j)\big) \Big].
  \end{multline*}
  Since $| A_n | = (n+1)!$ we show that the numerator counts the number of permutations for which $Y^{(i)}\tilde Y^{(j)} = 1$.
  We consider 6 different types of permutations $\pi \in A_n$.
  The following table lists a type of permutation together with the number of such permutations and the probability that $Y^{(i)}\tilde Y^{(j)} = 1$.
  \vspace{-3ex}
  \begin{center}
    \begin{tabular}{rclcl}
      $\big\{\pi(i),\pi(i+1)\big\} \cap \big\{j,j+1\big\} = \emptyset$ & : & $(n-1)(n-2)(n-1)!$ & $\cdot$ & $1/4$ \\
      $\big\{\pi(i),\pi(i+1)\big\} = \big\{j,j+1\big\}$ & : & $2(n-1)!$ & $\cdot$ & $1/2$ \\
      $\pi(i) = j, \quad \pi(i+1) \neq j+1$ & : & $(n-1)(n-1)!$ & $\cdot$ & $(i-1)(j-1)/(n-1)^2$ \\
      $\pi(i) = j+1, \quad \pi(i+1) \neq j$ & : & $(n-1)(n-1)!$ & $\cdot$ & $(n-i)(j-1)/(n-1)^2$ \\
      $\pi(i) \neq j, \quad \pi(i+1) = j+1$ & : & $(n-1)(n-1)!$ & $\cdot$ & $(n-i)(n-j)/(n-1)^2$ \\
      $\pi(i) \neq j+1, \quad \pi(i+1) = j$ & : & $(n-1)(n-1)!$ & $\cdot$ & $(i-1)(n-j)/(n-1)^2$
    \end{tabular}
  \end{center}
  The claim follows.
  Using that
  \[
    \sum_{i,j=1}^n (i-1)(j-1) = \sum_{i,j=1}^n (i-1)(n-j) = \sum_{i,j=1}^n (n-i)(j-1) = \sum_{i,j=1}^n (n-i)(n-j) = \binom{n}{2}^2,
  \]
  we obtain
  \begin{align*}
    2\sum_{i,j=1}^n \E(Y^{(i)}\tilde Y^{(j)})
      &= \frac{2}{(n+1)!}\left[ \tfrac{1}{4}n^2(n-1)(n-1)! + n^2\cdot (n-1)! + 4(n-2)!\binom{n}{2}^2 \right] \\
      &= \frac{2}{(n+1)!}\left[ \binom{n}{2}^2(n-2)(n-2)! + n\cdot n! + 4(n-2)!\binom{n}{2}^2 \right] \\
      &= \frac{2}{(n+1)!}\left[ \binom{n}{2}^2(n+2)(n-2)! + n\cdot n! \right] \\
      &= \frac{2n}{n+1}   \big( \tfrac{1}{4}(n-1)(n+2) + 1 \big) \\
      &= \frac{n}{n+1}    + \frac{n^2}{2}.\qedhere
  \end{align*}
\end{proof}

\begin{proposition}
\label{prop:Bdesides}
  The mean and variance of the distribution $X_{\des+\ides}$ on $B_n$ are
  \[
    \E(X_{\des+\ides}) = n,\qquad \V(X_{\des+\ides}) = 2\V(X_\des) + 1/2.
  \]
\end{proposition}

\begin{proof}
  The computation for the mean is obvious.
  For the variance, we follow the same argument as for $A_n$, except that we have to deal with more cases.
  The main step is again to analyze the mean of a summand $\E(Y^{(i)}\tilde Y^{(j)})$, using in particular \Cref{r:invdes}.
  We organize the summands into different cases which are presented as tables containing numbers of occurrences and probabilities.
  The caption of each table is one of the~$6$ mutually exclusive situations as for the symmetric group.
  Now each table has (at most) four rows indicating the special cases that $i=0$ or $j=0$ as follows:
  \[
    \shortstack{$i,j > 0$   \\ ++} \qquad
    \shortstack{$i = 0 < j$ \\ 0+} \qquad
    \shortstack{$i>0=j$     \\ +0} \qquad
    \shortstack{$i,j=0$     \\ 00}
  \]
  Rows for impossible situations are omitted.
  Every row contains in order the sign indicator, the number of signed permutations in this situation, and the probability that $Y_i\tilde Y_j = 1$.
  In cases 3--6, these probabilities also depend on the signs of $\pi(i), \pi(i+1), \pi^{-1}(j), \pi^{-1}(j+1)$.
  In these tables there are four columns with probabilities, labeled by  $\pm$-sequences. 

  \setlength{\tabcolsep}{3pt}

  \textbf{Case 1:} $\big\{|\pi(i)|,|\pi(i+1)|\big\} \cap \big\{|\pi^{-1}(j)|,|\pi^{-1}(j+1)|\big\} = \emptyset$:
  \begin{center}
  \begin{tabular}{r||l||c}
    ++ & $2^n \cdot 2\binom{n-2}{2}(n-2)!$ & $\tfrac{1}{4}$ \\
    \hline
    0+ & $2^n \cdot  \binom{n-2}{1}(n-1)!$ & $\tfrac{1}{4}$ \\
    \hline
    +0 & $2^n \cdot 2\binom{n-1}{2}(n-2)!$ & $\tfrac{1}{4}$
  \end{tabular}
  \end{center}

  \bigskip

  \textbf{Case 2:} $\big\{|\pi(i)|,|\pi(i+1)|\big\} = \big\{|\pi^{-1}(j)|,|\pi^{-1}(j+1)|\big\}$:
  \begin{center}
  \begin{tabular}{r||l||c}
    ++ & $2^n \cdot 2(n-2)!$ & $\tfrac{3}{8}$ \\
    \hline
    00 & $2^n \cdot (n-1)!$ & $\tfrac{1}{2}$
  \end{tabular}
  \end{center}

  \bigskip

  \textbf{Case 3:} $|\pi(i)| = j, \quad |\pi(i+1)| \neq j+1$:
  \begin{center}
  \small
  \begin{tabular}{r||l||c|c|c|c}
    && \footnotesize $+++\pm$ & \footnotesize $+-+\pm$ & \footnotesize $-+-\pm$ & \footnotesize $---\pm$ \\
    \hline
    \hline
    ++ & \footnotesize $2^{n-3}(n-2)(n-2)!$ &
        $\tfrac{j-1}{n-2}\big(\tfrac{i-1}{n-2}+1\big)$ &
        $1\cdot\big(\tfrac{i-1}{n-2}+1\big)$ &
        $0$ &
        $\tfrac{n-j-1}{n-2}\big(0+\tfrac{n-i-1}{n-2}\big)$ \\[5pt]
    \hline
    00 & \footnotesize $2^{n-3}(n-1)(n-1)!$ &
        $0$ & $1\cdot(0+1)$ & $0$ & $1\cdot(0+1)$
  \end{tabular}
  \end{center}

  \bigskip

  \textbf{Case 4:} $|\pi(i)| = j+1, \quad |\pi(i+1)| \neq j$:
  \begin{center}
  \small
  \begin{tabular}{r||l||c|c|c|c}
    && \footnotesize $++\pm+$ & \footnotesize $+-\pm+$ & \footnotesize $-+\pm-$ & \footnotesize $--\pm-$ \\
    \hline
    \hline
    ++ & \footnotesize $2^{n-3}(n-2)(n-2)!$ &
        \cellcolor{black!25}$\tfrac{j-1}{n-2}\big(\tfrac{n-i-1}{n-2}+0\big)$ &
        $1\cdot \big(\tfrac{n-i-1}{n-2}+0\big)$ &
        $0$ &
        $\tfrac{n-j-1}{n-2}\big(1+\tfrac{i-1}{n-2}\big)$ \\[5pt]
    \hline
    +0 & \footnotesize $2^{n-3}(n-1)!$ &
        $0$ & $1\cdot(0+0)$ & $0$ & $1\cdot(1+1)$
  \end{tabular}
  \end{center}

  \bigskip

  \textbf{Case 5:} $|\pi(i)| \neq j, \quad |\pi(i+1)| = j+1$:
  \begin{center}
  \small
  \begin{tabular}{r||l||c|c|c|c}
    && \footnotesize $++\pm+$ & \footnotesize $-+\pm+$ & \footnotesize $+-\pm-$ & \footnotesize $--\pm-$ \\
    \hline
    \hline
    ++ & \footnotesize $2^{n-3}(n-2)(n-2)!$ &
        $\tfrac{n-j-1}{n-2}\big(\tfrac{n-i-1}{n-2}+0\big)$ &
        $0$ &
        $1\cdot\big(1+\tfrac{i-1}{n-2}\big)$ &
        $\tfrac{j-1}{n-2}\big(1 + \tfrac{i-1}{n-2}\big)$ \\[5pt]
    \hline
    0+ & \footnotesize $2^{n-3}(n-1)!$ &
        $0$ &
        $0$ &
        $1\cdot(0+1)$ &
        $1\cdot(0+1)$ \\[5pt]
    \hline
    +0 & \footnotesize $2^{n-3}(n-1)!$ &
        $0$ &
        $0$ &
        $1\cdot(1+1)$ &
        $0$
  \end{tabular}
  \end{center}

  \bigskip

  \textbf{Case 6:} $|\pi(i)| \neq j, \quad |\pi(i+1)| = j+1$:
  \begin{center}
  \small
  \begin{tabular}{r||l||c|c|c|c}
    && \footnotesize $+++\pm$ & \footnotesize $-++\pm$ & \footnotesize $+--\pm$ & \footnotesize $---\pm$ \\
    \hline
    \hline
    ++ & \footnotesize $2^{n-3}(n-2)(n-2)!$ &
        $\tfrac{n-j-1}{n-2}\big(\tfrac{i-1}{n-2}+1\big)$ &
        $0$ &
        $1\cdot\big(0+\tfrac{n-i-1}{n-2}\big)$ &
        $\tfrac{j-1}{n-2}\big(0+\tfrac{n-i-1}{n-2}\big)$ \\[5pt]
    \hline
    0+ & \footnotesize $2^{n-3}(n-1)!$ &
        $0$ & $0$ & $1\cdot\big( 1 + 0\big)$ & $1\cdot\big( 1 + 0\big)$
  \end{tabular}
  \end{center}
  We discuss one entry in detail to illustrate how to read these tables.
  Consider the highlighted situation $i,j>0$ with $\pi(i),\pi(i+1),\pi^{-1}(j+1) > 0$ in \textbf{Case~4}.
  The two possible signs for $\pi^{-1}(j)$ are treated separately and correspond to the sum in the entry.
  That is, for $\pi^{-1}(j) > 0$ the probability is $\tfrac{j-1}{n-2}\cdot\tfrac{n-i-1}{n-2}$, while for $\pi^{-1}(j) < 0$ the probability is $\tfrac{j-1}{n-2}\cdot0$.

  First, we count signed permutations in this case, treating absolute value and signs individually.
  The value $|\pi(i)| = j+1$ is fixed, and $|\pi(i+1)| \neq j$ means that there are $n-2$ choices for the absolute value of $\pi(i+1)$ and $(n-2)!$ choices for the absolute values of $\{\pi(k) \mid k \neq i,i+1\}$.
  Four signs are fixed by the column label, but since $|\pi(i)| = j + 1$, the signs of $\pi(i)$ and $\pi^{-1}(j+1)$ coincide, giving a total of $n-3$ signs which can be chosen freely, giving in total $2^{n-3}$ possible sign configurations for the remaining entries.

  Second, the probability that~$i$ is a descent is $\tfrac{j-1}{n-2}$ since $\pi(i) = j+1,\pi(i+1) > 0$ and $\pi(i+1) \neq j$ leaving $j-1$ possible values for $\pi(i+1)$ out of $n-2$ in total.

  Third, we consider the two possibilities for the sign of~$\pi^{-1}(j)$.
  The probability that $i$ is a descent is independent of this because $|\pi(i+1)| \neq j$.
  If $\pi^{-1}(j) > 0$, we have, according to \Cref{r:invdes}, that~$j+1$ must be to the left of~$j$.
  Since~$j+1$ is in position~$i$, and~$j$ cannot be in position $i+1$, there are $n-i-1$ positions to the right, out of $n-2$ positions in total.
  If $\pi^{-1}(j) < 0$, than~$j$ cannot be an inverse descent since this situation does not appear as a possibility in \Cref{r:invdes}.

  In total, a random signed permutation in this situation has a descent in position~$i$ and an inverse descent in position~$j$ with probability
  \[
    2^{n-3}(n-2)(n-2)!\ \tfrac{j-1}{n-2}\big(\tfrac{n-i-1}{n-2}+0\big).
  \]

  \bigskip

  Summing all $6$ cases individually for $0 \leq i,j < n$, and then summing the cases yields
  \begin{multline*}
    2^{n-2}(n-1)!(n-1)\big((n-2)(n-3)+2(n-2)\big) +
    2^{n-2}(n-1)!(3n-1) + \\
    2^{n-4}(n-1)(n-1)!(5n-6) +
    2^{n-4}(n-1)!(n-1)(3n-2) + \\
    2^{n-4}(n-1)(n-1)!(5n-2) +
    2^{n-4}(n-1)!(n-1)(3n-2) = \\
    2^{n-2} \cdot n!\cdot (n^2 + 1),
  \end{multline*}
   giving in total
   \begin{align*}
     2\sum_{i,j=0}^{n-1} \E(Y^{(i)}\tilde Y^{(j)}) &= \frac{1}{2^{n-1}\cdot n!} \cdot 2^{n-2} \cdot n! \cdot (n^2+1) = \frac{n^2+1}{2} = \frac{n^2}{2} + \frac{1}{2}. \qedhere
   \end{align*}
\end{proof}

\begin{proposition}
\label{prop:Ddesides}
  The mean and variance of the distribution $X_{\des+\ides}$ on $D_n$ are
  \[
    \E(X_{\des+\ides}) = n,\qquad \V(X_{\des+\ides}) = 2\V(X_\des) + {n}/{(2n-2)}.
  \]
\end{proposition}

\begin{proof}
  The computation for the mean is obvious.
  This time, we have to show that
  \[
    2\sum_{i,j=0}^{n-1} \E(Y^{(i)}\tilde Y^{(j)}) - \frac{n^2}{2} = \frac{n}{2n-2}.
  \]
  This can be obtained from the variance in type~$B$ as follows.
  Even though we follow the convention $\pi(0) = -\pi(2)$ for computing descents in type~$D$, we follow the type~$B$ convention to distinguish the cases.
  That is, we let $\pi(0) = 0$ in the case distinction.
  One can check that except for three situations listed below, one obtains the same probabilities, but half the counts compared to type~$B$ (since $D_n$ is an index~$2$ subgroup of~$B_n$).
  The three exceptions are the following replacements
  \begin{center}
    \begin{tabular}{rlcl}
      situation $0+$ in Case 6: & $2^{n-2}(n-1)!$ & $\rightsquigarrow$ & $2^{n-3}\cdot n(n-2)!$ \\
      situation $+0$ in Case 4: & $2^{n-2}(n-1)!$ & $\rightsquigarrow$ & $2^{n-3}\cdot n(n-2)!$ \\
      situation $00$ in Case 2: & $2^{n-1}(n-1)!$ & $\rightsquigarrow$ & $2^{n-3}\cdot n(n-2)!$\ .
    \end{tabular}
  \end{center}
  Here, each situation is meant as the total contribution of this complete row in the above table.
  This is,
  \begin{align*}
    2^{n-2}(n-1)! &= 2^{n-3}(n-1)!\cdot\big(0+0+1(1+0)+1(1+0)\big) \\
                  &= 2^{n-3}(n-1)!\cdot\big(0+1(0+0)+0+1(1+1)\big) \\
    2^{n-1}(n-1)! &= 2^{n}(n-1)!\cdot\tfrac{1}{2}
  \end{align*}
  We explain this in Case 2, the others being similar.
  In type~$B_n$ in this situation and case, $\pi(1)$ is determined by $j+1$, so there are $(n-1)!$ permutations left, together with $2^{n-1}$ signs that yield a descents and an inverse descent at the same time.
  On the other hand, in type~$D_n$, one has to check that both $\pi(2) < 0$ and $\pi^{-1}(2) < 0$.
  So one either has $|\pi(2)| = 2$ and obtains $(n-2)!$ permutations and $2^{n-2}$ possible signs, or one has $|\pi(2)| \neq 2$ and has $(n-2)(n-2)!$ permutation and $2^{n-3}$ possible signs.
  Summing these yields
  \[
    2^{n-2}(n-2)! + 2^{n-3}(n-2)(n-2)! = 2^{n-3}\cdot n(n-2)!\ .
  \]
  Observing that the situation $00$ occurs once, while each of the situations $0+$ and $+0$ occurs $n-1$ times, we obtain
  \begin{align}
    \label{eq:BtoD}
    \begin{aligned}
      2^{n-1}(n-1)! + 2(n-1)\cdot 2^{n-2}(n-1)! &= 2^{n-1}\cdot n! \\
      2^{n-3}\cdot n(n-2)! + 2(n-1)\cdot 2^{n-3}\cdot n(n-2)! &= 2^{n-2}\cdot n! + 2^{n-3}\cdot n(n-2)!\ .
    \end{aligned}
  \end{align}
  We are thus ready to deduce the proposition.
  Let
  \[
    S_B = 2^{n-2}\cdot n!\cdot(n^2+1)
  \]
  be the formula from the proof in type~$B_n$.
  Then the analogous formula in type~$D_n$ is
  \begin{align*}
    S_D &= \big( S_B - 2^{n-1}\cdot n! \big)/2 + 2^{n-2}\cdot n! + 2^{n-3}\cdot n(n-2)! \\
        &= 2^{n-3}\cdot n!\cdot(n^2+1) + 2^{n-3}\cdot n(n-2)! \\
        &= 2^{n-3}\cdot n(n-2)!\big( (n-1)(n^2+1) + 1 \big) \\
        &= 2^{n-3}\cdot n(n-2)!\big( n^2(n-1) + n \big).
  \end{align*}
  We finally compute
  \begin{align*}
    2\sum_{i,j=0}^{n-1} \E(Y^{(i)}\tilde Y^{(j)})
      &= \frac{1}{2^{n-2}\cdot n!} \cdot 2^{n-3}\cdot n(n-2)!\big( n^2(n-1) + n \big) \\
      &= \frac{1}{2(n-1)}\big(n^2(n-1) + n\big) \\
      &= \frac{n^2}{2} + \frac{n}{2n-2}. \qedhere
  \end{align*}
\end{proof}

\begin{proof}[Proof of \Cref{thm:desides} and of \Cref{cor:desides}]
  The classical types are dealt with in \Cref{prop:Adesides,,prop:Bdesides,,prop:Ddesides}.
  The computation in the dihedral types~$I_2(m)$ is obvious, and the remaining were computed using
  \sage~\cite{sagemath}.
\end{proof}

\section{Limit theorems}
\label{sec:clts}

We finally turn to the limit theorems for Mahonian and Eulerian distributions of sequences of Coxeter groups of increasing rank.
These depend only very mildly on the concrete sequence of finite Coxeter groups in the sense that only the maximal sizes of dihedral parabolic subgroups play a role, see \Cref{rem:Wmahoniandist,,rem:Weuleriandist} and \Cref{cor:mmaxbounded}.

\medskip

For each $n\in\NN$, let $X^{(n)}$ be a real valued random variable with cumulative distribution function $F_n(x) = \Prob(X^{(n)} \le x)$, and let~$D$ be a distribution with cumulative distribution function~$F$.
The sequence $X^{(n)}$ \defn{converges in distribution} to~$D$, denoted $X_{n}\dto D$, if $F_n(x) \longrightarrow F(x)$ for all $x\in\RR$ where $F$ is continuous.
Denote the standard normal distribution by~$N(0,1)$.
The sequence $X^{(n)}$ \defn{satisfies the CLT} if, for $n \rightarrow \infty$,
\[
  \frac{X^{(n)} - \E(X^{(n)})}{\sqrt{\V(X^{(n)})}} \dto N(0,1).
\]

\medskip

Set $X_\inv{(W)}$ and $X_\des{(W)}$ to be, respectively, the Mahonian distribution and the Eulerian distribution on a finite Coxeter group~$W$.

\begin{theorem}
\label{thm:WCLTinv}
  Let $W^{(1)},W^{(2)},\ldots$ be an infinite sequence of finite Coxeter groups such that~$W^{(n)}$ has rank~$n$ and maximal degree $d_n$.
  Then $X_{\inv}(W^{(n)})$ with variance $s_n^2$ satisfies the CLT if and only if
  \[
    d_n / s_n \longrightarrow 0 \quad \text{for } n \rightarrow \infty.
  \]
\end{theorem}

\begin{theorem}
\label{thm:WCLTdes}
  Let $W^{(1)},W^{(2)},\ldots$ be an infinite sequence of finite Coxeter groups such that~$W^{(n)}$ has rank~$n$.
  Then $X_{\des}(W^{(n)})$ with variance $s_n^2$ satisfies the CLT if and only if
  \[
    s_n \longrightarrow \infty \quad \text{for } n \rightarrow \infty.
  \]
\end{theorem}

To structure the proofs of the theorems we separate general arguments from probability theory in \Cref{sec:generalargs} from concrete statements using properties of finite Coxeter groups in \Cref{sec:concreteargs}.
Preceeding these proofs, we discuss the conditions in both theorems in detail, give examples of sequences of Coxeter groups which fulfil or violate them, and provide a local limit theorem in \Cref{thm:localLimitTheorem}.

\bigskip

For functions $f,g: \NN_{+} \to \RR_{\ge 0}$, we use \defn{big-$\bigO$-notation} $f(n)\in \bigO(g(n))$, if there exists $c > 0$ and an~$N \in \NN$ such that for all $n \geq N$, we have $f(n) \leq cg(n)$, and we use \defn{little-$\smallO$-notation} $f(n)\in \smallO(g(n))$, if for all $c > 0$ there exists an~$N \in \NN$ with this property.
We often use the equivalence $f(n)\in\smallO(g(n)) \Leftrightarrow f(n)/g(n) \longrightarrow 0$.

\begin{proposition}
  \label{rem:Wmahoniandist}
  In the notation of \Cref{thm:WCLTinv}, the condition $d_{n}/s_{n}\longrightarrow 0$ is equivalent to the condition $m_{n}/s_{n} \longrightarrow 0$ where $m_{n} = \mmax (W^{(n)})$ is half the maximal size of a dihedral parabolic subgroup of~$W^{(n)}$.
\end{proposition}

\begin{proof}
  We have $2 \leq m_n \leq d_n$.
  The first inequality is by definition.
  The second follows from inspection for the irreducible Coxeter groups.
  It is true for any finite Coxeter group, since the degrees of a reducible Coxeter group are the multiset union of the degrees of its irreducible components.
  This implies the forward implication
  \[
    d_{n}/s_{n}\longrightarrow 0\ \Rightarrow\ m_{n}/s_{n} \longrightarrow 0.
  \]

  For the reverse implication we use the following observation:
  For any infinite subsequence $n_k$ such that $m_{n_k} < d_{n_k}$ for all~$k$, we have $d_{n_k} / s_{n_k} \longrightarrow 0$.
  This is because if $m_{n_k} < d_{n_k}$ the maximal degree $d_{n_k}$ cannot come from an irreducible dihedral component and thus $d_{n_k} \in o(s_{n_k})$, as $s_{n_{k}}^{2} \ge \sum_{i=1}^{d_{n_{k}}} i^{2}$ and $s_{n_{k}} \to \infty$, by \Cref{cor:Winversions}.

  Assume that $d_{n}/s_{n}\centernot\longrightarrow 0$.
  Then there exists an $\epsilon > 0$ and a subsequence $n_k$ with $d_{n_k} / s_{n_k} > \epsilon$ for all~$k$.
  In particular, this subsequence does not have any further subsequences that converge to~$0$.
  The above observation implies that the sequence $n_{k}$ can contain only finitely many indices~$k$ with $m_{n_k} < d_{n_k}$, and thus $m_{n_k} = d_{n_k}$ for all large enough~$k$.
  Therefore $m_{n_k} / s_{n_k} \centernot\longrightarrow 0$ and also $m_{n}/s_{n} \centernot\longrightarrow 0$.
\end{proof}

In the following proposition, by the \defn{non-dihedral component} of a finite Coxeter group~$W$ we mean the parabolic subgroup of~$W$ containing all irreducible components of~$W$ that are not of dihedral type.

\begin{proposition}
\label{rem:Weuleriandist}
In the notation of \Cref{thm:WCLTdes}, the conditions
\begin{itemize}
  \item[(A1)] the rank of the non-dihedral component of $W^{(n)}$ tends to infinity,
  \item[(A2)] the rank of the non-dihedral component of $W^{(n)}$ is not globally bounded,
  \item[(B)] the irreducible dihedral components $\big\{I_2\big(m_i^{(n)}\big)\big\}_{i\in I(n)}$ of $W^{(n)}$ satisfy
  \[
  \sum_{i\in I(n)} \frac{1}{m_i^{(n)}} \longrightarrow \infty.
  \]
\end{itemize}
satisfy\,\footnote{In the published version of this paper this wrongly reads $s_n \longrightarrow \infty\ \Leftrightarrow\ \big[ (\mathrm{A2}) \text{ or }(\mathrm{B}) \big]$.  We thank Benjamin Brück and Frank Röttger for pointing this out.}
  \[
  \big[ (\mathrm{A1}) \text{ or } (\mathrm{B}) \big]\ \Rightarrow\ s_n \longrightarrow \infty\ \Rightarrow\ \big[ (\mathrm{A2}) \text{ or } (\mathrm{B}) \big].
  \]
\end{proposition}
\begin{proof}
  We employ \Cref{cor:Wdescents}.
  Clearly, (A1) and (B) are both sufficient conditions for $s_n \longrightarrow \infty$.
  Then assume $s_{n}\longrightarrow\infty$.
  If the rank of the non-dihedral component is globally bounded, then the growth of $s_{n}$ is determined by the irreducible dihedral components whose variance sum must diverge as in (B).
\end{proof}

\Cref{rem:Wmahoniandist,,rem:Weuleriandist} can be applied to known sequences of finite Coxeter groups, for example, yielding CLTs for sequences of Weyl groups.
The proof of the following corollary is in \Cref{sec:concreteargs}.

\begin{corollary}
  \label{cor:mmaxbounded}
  Let $W^{(1)},W^{(2)},\ldots$ be an infinite sequence of finite Coxeter groups such that~$W^{(n)}$ has rank~$n$ and such that the maximal size of dihedral parabolic subgroups of all $W^{(n)}$ is globally bounded.
  Then $X_\inv{(W^{(n)})}$ and $X_\des{(W^{(n)})}$ satisfy CLTs.
  In particular this holds for any sequence of finite Weyl groups.
\end{corollary}

\begin{remark}
  The condition that the rank of $W^{(n)}$ equals $n$ in \Cref{thm:WCLTinv,,thm:WCLTdes} and \Cref{cor:mmaxbounded} may be relaxed to the condition that $W^{(1)},W^{(2)},\dots$ is an infinite sequence of finite Coxeter groups of increasing rank.
  To prove this generalization one needs to work with the more general version of \Cref{thm:LindebergFeller} that is discussed in the provided references.
  We use this mild generalization only in the following example.
\end{remark}

\begin{example}
  The following four situations show the various possibilities of CLTs for Mahonian and Eulerian distributions, where we set $X_\inv^{(n)} = X_\inv(W^{(n)}), X_\des^{(n)} = X_\des(W^{(n)})$ and $m_{n} = \mmax (W^{(n)})$.
  \begin{enumerate}[(1)]
    \item Let $W^{(n)} = \prod_{i=1}^n I_2(i)$ so that $m_{n} = n$.
    For $X_{\inv}$ we have $s_{n}^2 \sim \sum_{i=1}^{n} i^2 \sim n^3$ and, by \Cref{rem:Wmahoniandist}, $X_\inv^{(n)}$ satisfies the CLT.
    For $X_\des$ we have $s_{n}^2 \sim \sum_{i=1}^{n} \frac{1}{i} \longrightarrow \infty$, so $X_\des$ also satisfies the CLT.

    \item Let $W^{(n)} = \prod_{i=1}^n I_2(i^2)$, so that $m_{n} = n^2$.
    For $X_{\inv}$ we have $s_{n}^2 \sim \sum_{i=1}^{n} i^4 \sim n^5$ and $X_\inv$ satisfies the CLT.
    For $X_\des$ we have $s_{n}^2 = \sum_{i=1}^{n} \frac{1}{i^2} \longrightarrow \pi^2/6$, so $X_{\des}$ does not satisfy the CLT.

    \item Let $W^{(n)} = A_{1}^{n-2} \times I_{2}(n)$ so that $m_{n}= n$.
    For $X_\inv$ we have $s_n^{2} \sim n^{2}$, so $X_\inv$ does not satisfy the CLT.
    For $X_\des$ we have $s_n^2 \sim n \longrightarrow \infty$, so $X_\des$ satisfies the CLT.
    
    \item Let $W^{(n)} = \prod_{i=1}^n I_2(2^i)$ so that $m_n = 2^{n}$.
    For $X_\inv$ we have $s_{n}^{2} \sim \sum_{i=1}^{n} 2^{2i} \sim 2^{2n}$ and $X_{\inv}$ does not satisfy the CLT.
    For $X_\des$ we have $s_n^2 = \sum_{i=1}^{n} \frac{1}{2^i} \longrightarrow 1$, so $X_\des$ does not satisfies the CLT.
  \end{enumerate}
\end{example}

The central limit theorem gives only a qualitative feel for the behavior of the distributions of $X_\inv$ and $X_\des$.
Following Bender~\cite{bender1973central}, however, we can lift the central limit theorems to the stronger uniform convergence of the probabilities $\Prob(X^{(n)}_\inv = k)$ and $\Prob(X^{(n)}_\des = k)$ to the density of the normal distribution.

\begin{corollary}
\label{thm:localLimitTheorem}
  Let $X^{(n)}$ denote either the Mahonian distribution from \Cref{thm:WCLTinv} or the Eulerian distribution from \Cref{thm:WCLTdes}.
  If $X^{(n)}$ satisfies the CLT then
  \[
    \lim_{n\to\infty}\ \sup_{x\in\RR}\ \left| s_np_n(\lfloor s_n x +
      \mu_n\rfloor) - \frac{1}{\sqrt{2\pi}} e^{-x^2/2} \right| = 0
  \]
  where $p_n(k) = \Prob(X^{(n)} = k)$,
  $s_n^2 = \V(X^{(n)})$ and $\mu_n = \E(X^{(n)})$.
  Furthermore the rate of convergence depends only on~$s_n$ and the rate of convergence in \Cref{thm:WCLTinv,,thm:WCLTdes}.
\end{corollary}

\begin{remark}
  One might be able to strengthen the convergence in \Cref{thm:localLimitTheorem} to a mod-Gaussian convergence in the sense of~\cite{FMN2016}.
  For this one in particular needs to consider also the fourth cumulants of the Mahonian and Eulerian distributions.
  For the $W$-Mahonian distribution one obtains a mod-Gaussian convergence in all classical types.
  With $\alpha_n = \beta_n = n$ in~\cite[Chapter~5.1]{FMN2016} one computes
  \[
    \kappa_2(X^{(n)}) = \sigma^2 n^3 (1 + O(n^{-1})), \quad
    \kappa_4(X^{(n)}) = L n^5 (1 + O(n^{-1}))
  \]
  for some constants $\sigma,L$, as needed for the mod-Gaussian convergence.
  Analogously, for the $W$-Eulerian distribution, one can use $\alpha_n = n$ and $\beta_n = 1$ and derive the needed property for~$\kappa_2(X^{(n)})$.
  The computations for $\kappa_4(X^{(n)})$ might possibly be achieved in the same way as the computation for $\kappa_2(X^{(n)})$ in \Cref{sec:Wdescents}.
\end{remark}

Chatterjee and Diaconis have shown a CLT for the double-Eulerian distribution on $W^{(n)}=\Perm_n$~\cite{chatterjee2016central}.
The $W$-double-Eulerian analogues of the above theorems are open.
Since the first posting of this paper, some progress on the following problem has been made by Röttger in~\cite{rottger2018asymptotics}.
\begin{problem}
 Find necessary and sufficient conditions on general sequences of finite Coxeter groups of increasing rank under which the double-Eulerian distribution satisfies a CLT.
\end{problem}

\subsection{Conditions for limit theorems}
\label{sec:generalargs}

A \defn{triangular array} is a set of random variables $X^{(n,i)}$ with $i = 1,\dots, n$ for $n=1,2,\dots$,
such that for fixed~$n$ the random variables $X^{(n,i)}$ are independent with nonzero finite variances $0 < \V(X^{(n,i)}) < \infty$.
A triangular array of random variables satisfies the \defn{maximum condition} if
\[
  \max_i\{\V(X^{(n,i)})\} \big/ \V(X^{(n)}) \longrightarrow 0,
\]
where we set $X^{(n)} = \sum_i X^{(n,i)}$.
It satisfies the \defn{Lindeberg condition} if, for all $\epsilon > 0$,
\[
  \frac{1}{s_n^2}\sum_{i=1}^n\mathbb{E}\big((X^{(n,i)})^2\cdot I\big\{|X^{(n,i)}| \geq \epsilon s_n\big\}\big) \longrightarrow 0
\]
where $s_n^2 = \sum_i \V(X^{(n,i)})$ is the variance of $X^{(n)} = \sum_i X^{(n,i)}$, and where~$I\{\cdot\}$ is the indicator function.

\medskip

The following theorem goes back to the work of Lindeberg and Feller in the first half of the 20th century.
See \cite[Theorem~15.43]{klenke2013probability}~and \cite[Sections~27 and~28]{billingsley2008probability} for details.

\begin{theorem}[Lindeberg--Feller theorem for triangular arrays]
  \label{thm:LindebergFeller}
  Let $X^{(n,i)}$ be a triangular array of random variables, and let $X^{(n)} = X^{(n,1)} +\cdots + X^{(n,n)}$.
  Then $X^{(n)}$ satisfies the Lindeberg condition if and only if it satisfies the CLT and the maximum condition.
\end{theorem}

The following proposition is the key ingredient in the proof of \Cref{thm:WCLTinv}.
\begin{proposition}
\label{prop:uniformvarCLT}
  For each $n \in \NN_+$, fix integers $2 \leq d_{n,1}\leq \dots\leq d_{n,n}$.
  Let $X^{(n,i)}$ be independent random variables, each uniformly distributed on $\{0,1,\dots,d_{n,i}-1\}$.
  Then $X^{(n)} = \sum_{i=1}^n X^{(n,i)}$ satisfies the CLT if and only if it satisfies the maximum condition.
\end{proposition}

The maximum condition in this proposition has the following convenient reformulation.

\begin{lemma}
\label{lem:maxcondinv}
  In the notation of \Cref{prop:uniformvarCLT}, we have that $X^{(n)}$ satisfies the maximum condition if and only if $d_{n,n} \in o(s_n)$ for $s_n^2 = \V(X^{(n)})$.
\end{lemma}

\begin{proof}
  We have $\V(X^{(n,i)}) = (d_{{n,i}}^{2}-1)/12$.
  The maximum condition is thus equivalent to $(d_{n,n}^{2}-1)/s_{n}^{2} \longrightarrow 0$.
  Since $d_{n,i} \geq 2$ for all~$n$ and all $1 \leq i \leq n$, we have that $s_{n} \longrightarrow \infty$ and the maximum condition is equivalent to $d_{n,n} \big / s_n \longrightarrow 0$.
\end{proof}

\begin{proof}[Proof of \Cref{prop:uniformvarCLT}]
  Assume first the maximum condition.
  By \Cref{lem:maxcondinv}, for any $\epsilon > 0$ there exists an~$N$ such that for all $n>N$, $\epsilon s_{n} > d_{n,n}$, where we denote, as usual, $s_n^2 = \V(X^{(n)})$.
  Because $\Prob( X^{(n,i)} \ge d_{n,n} ) = 0$ the Lindeberg condition holds, as for these~$n$
  \[
    \E\big((X^{(n,i)})^2\cdot I\big\{|X^{(n,i)}| \geq \epsilon s_n\big\}\big) = 0.
  \]
  The CLT then holds by \Cref{thm:LindebergFeller}.

  For the reverse implication we first compute
  the fourth and sixth cumulant as
  \[
    - \kappa_4(X^{(n)}) = 
    \frac{1}{120}\sum_{i=1}^n(d_{n,i}^4-1)
    \quad \text{and} \quad
    \kappa_{6}(X^{(n)}) = \frac{1}{252} \sum_{i=1}^{n}(d_{n,i}^{6}-1).
  \]
  This implies that $-1 \le \kappa_{k}(X^{(n)}/s_{n}) = \kappa_{k}(X^{(n)})/s_{n}^{k}\le 1$ for $k\le 6$ since $s_{n}^{k}$ contains each $(d_{n,i}^{k}-1)$ as a summand and the odd cumulants vanish.  Since the $k$-th moment is a polynomial in the first $k$ cumulants, this implies that the sixth moment is bounded.
  Assuming the CLT, \cite[Theorem~25.12]{billingsley2008probability} yields that the first four central moments of $X^{(n)}/s_{n}$ converge to those of $N(0,1)$.
  Consequently $\kappa_{4}(X^{n}/s_{n}) = \kappa_{4}(X^{(n)})/s_{n}^{4} \longrightarrow 0$ and thus $d_{n,n}/s_{n} \longrightarrow 0$.
  By \Cref{lem:maxcondinv} this is the maximum condition.
\end{proof}

The following two propositions are the key ingredients in the proof of \Cref{thm:WCLTdes}.
\begin{proposition}
\label{prop:finitesupportCLT}
Let $X^{(n,i)}$ be a triangular array of globally bounded random variables such that $\V(X^{(n)}) \longrightarrow \infty$.
Then $X^{(n)}$ satisfies the CLT.
\end{proposition}

\begin{proof}
  Let $C$ be such that the $\Prob(|X^{(n,i)}| > C) = 0$ for all $n$ and all $1\le i \le n$, and let $\epsilon > 0$ be arbitrary.
  Since $s_{n}^2 = \V(X^{(n)}) \longrightarrow\infty$, there exists an~$N$ such that for all $n>N$, $\epsilon s_{n} > C$.
  Thus the Lindeberg condition holds.
\end{proof}

\begin{proposition}
  \label{prop:latticeCLT}
  Let $X^{(n)}$ be a sequence of random variables such that $X^{(n)} - \E(X^{(n)})$ takes values in a fixed lattice $\delta \ZZ \subset \RR$ for some $\delta>0$.
  If $X^{(n)}$ satisfies the CLT, then $\V(X^{(n)})\longrightarrow\infty$ as $n\to\infty$.
\end{proposition}
\begin{proof}
  Since $X^{(n)}-\E(X^{(n)})$ does not take values strictly between $0$ and $\delta$, we obtain
  \[
    \Prob \left( 0 < \frac{X^{(n)}-\E(X^{(n)})}{s_{n}} < \delta / s_{n} \right) = 0.
  \]
  Assume $s_{n}^{2}\centernot\longrightarrow\infty$.
  Then the sequence $s_{n}$ has a subsequence $s_{n_{m}}$ bounded by $s<\infty$, implying $\delta/s_{n_{m}} > \delta / s$ for all~$m$.
  Consequently, the cumulative distribution functions $F_{n}(x) = \Prob \left((X^{(n)}-\E(X^{(n)}))/s_{n} \le x\right)$ satisfy $F_{n_{m}}(0) = F_{n_{m}}(\delta/s)$ for all $m$.
  Since the cumulative distribution function of $N(0,1)$ is strictly increasing, it cannot be the pointwise limit of $F_{n_{m}}$ and thus not the pointwise limit of~$F_{n}$.
  Therefore the CLT does not hold.
\end{proof}

\subsection{Proofs of \Cref{thm:WCLTinv,,thm:WCLTdes} and \Cref{cor:mmaxbounded,,thm:localLimitTheorem}}
\label{sec:concreteargs}

To construct appropriate triangular arrays for the Mahonian and the Eulerian distributions, we make use of the factorizations~\eqref{eq:GFWinversions} and~\eqref{eq:GFWdescdents}.
Let $W^{(n)}$ be a finite Coxeter group of rank~$n$ with degrees $d^{(n)}_1 \leq \dots \leq d^{(n)}_n $, and let $q^{(n)}_1,\dots,q^{(n)}_n$ denote the negatives of the roots of the descent generating function.

\medskip

Given two polynomials $f,g\in\NN[z]$, one has $X_{fg} = X_{f} + X_{g}$, as independent random variables.
For inversions define independent random variables $X_\inv^{(n,i)}$ with uniform distribution on $\{0,1,\dots,d_i^{(n)}-1\}$.
Because of the factorization of $\GF_\inv (W^{(n)})$, we have
\begin{equation}
\label{eq:invdecomp}
X_\inv^{(n)} = X_{\inv}(W^{(n)}) = X_\inv^{(n,1)} + \cdots + X_\inv^{(n,n)}.
\end{equation}
Similarly, define independent Bernoulli random variables
\[
  X_\des^{(n,i)} =
  \begin{cases}
    0 & \text{with probability } \frac{q_i^{(n)}}{1+q_i^{(n)}}, \\
    1 & \text{with probability } \frac{1}{1+q_i^{(n)}}.
  \end{cases}
\]
Because of the factorization of $\GF_\des (W^{(n)})$, we have
\begin{equation}
\label{eq:desdecomp}
X_\des^{(n)} = X_{\des}(W^{(n)}) = X_\des^{(n,1)} + \cdots + X_\des^{(n,n)}.
\end{equation}

\begin{proof}[Proof of \Cref{thm:WCLTinv}]
  Use the decomposition~\eqref{eq:invdecomp} into a sum of discrete uniform distributions.
  The equivalence follows from \Cref{prop:uniformvarCLT} and \Cref{lem:maxcondinv} using the degrees $2 \leq d_{n,1} \leq \dots \leq d_{n,n}$ of~$W^{(n)}$.
\end{proof}

\begin{proof}[Proof of \Cref{thm:WCLTdes}]
  For the forward implication we use \Cref{prop:latticeCLT} with $\delta=1/2$ as $X^{(n)}_{\des}$ takes integer values and has mean~$n/2$.
  For the reverse implication use the decomposition~\eqref{eq:desdecomp} into sums of independent Bernoulli random variables and \Cref{prop:finitesupportCLT}.
\end{proof}

\begin{proof}[Proof of \Cref{cor:mmaxbounded}]
  For the Mahonian distribution this follows using \Cref{rem:Wmahoniandist} since $m_{n}$ is globally bounded.
  For the Eulerian distribution, if the dihedral part is bounded in size, the non-dihedral part is not bounded in rank and thus \Cref{rem:Weuleriandist} yields the sufficient condition for \Cref{thm:WCLTdes}.
\end{proof}

\begin{proof}[Proof of \Cref{thm:localLimitTheorem}]
  Given \Cref{thm:WCLTinv,,thm:WCLTdes}, this is \cite[Lemma~2]{bender1973central} and the log-concavity from \Cref{thm:Winvfactor,,thm:Wdesfactor}.
\end{proof}

\appendix
\section{Additional computational data}
\label{sec:findstat}

In this section, we present experimental investigations of the asymptotics of permutation statistics.
Assume one has computed explicit values of a permutation statistic $\stat : \Perm_n \longrightarrow \NN$ for $2 \leq n \leq N$ for some~$N$ (in our case typically~$6, 7,$ or~$8$).
One can then
\begin{enumerate}[(1)]
  \item compute the generating functions $\GF_{\stat}(z)$, mean and variance of the random variable $X_\stat$ for $2 \leq n \leq N$, and

  \item use Lagrange interpolation on the $N-1$ data points to guess (Laurent) polynomial formulas for the mean and variance of $X_\stat$ as a function of~$n$.
\end{enumerate}

As of February 2018, the database \texttt{www.FindStat.org}~\cite{FindStat} contains 1113 combinatorial statistics, including $285$ permutation statistics.
We have applied the above procedure to all these permutation statistics and searched for statistics $\stat: \Perm_n \longrightarrow \NN$ such that the variance of the random variable $X_\stat^{(n)}$ is of the form
$
  \V(X_\stat^{(n)}) = f(n) \big/ (an+b)^c
$
with $a,b \in \{0, \pm 1,\pm 2\}$ and $c \in \{0,1,2,3,4,5\}$ and polynomial~$f \in \QQ[n]$ such that the Lagrange interpolation had at least three more data points than the degree of~$f$.

\medskip

Among the $285$ permutation statistics, there are $14$ Mahonian statistics and $13$ Eulerian statistics.
On top of these we found additional statistics for which the Lagrange interpolation suggest variances of the above form and we list them below.
Every table contains in its headline all statistics that yield one fixed random variable $X_\stat^{(n)}$ followed by the interpolated mean and variance for that random variable.
Below we list numerical values for higher cumulants $\tilde\kappa^{(n)}_k = \tilde\kappa_k(X_\stat^{(n)}) = \kappa_k(X_\stat^{(n)} / s_n)$ normalized by $s_n = \kappa_2(X_\stat^{(n)})^{1/2}$.
To read this numerical information, recall that, assuming bounded moments, $X_{\stat}$ satisfies the CLT if and only if for all $k \geq 3$, one has $\tilde\kappa_k^{(n)}\longrightarrow 0$ as $n \rightarrow \infty$.

\medskip

Some of these distributions are well-known (e.g~the number of fixed points \href{http://www.findstat.org/St000022}{\tt St000022}) and some are not hard to compute (such as the sum of the descent tops \href{http://www.findstat.org/St000111}{\tt St000111} or the sum of the descent bottoms \href{http://www.findstat.org/St000154}{\tt St000154}).
Others seem unexpected at first glance (such as eigenvalues, indexed by permutations, of the random-to-random operator acting on the regular representation \href{http://www.findstat.org/St000500}{\tt St000500}).
Finally, the computational data suggests central limit theorems for several of these statistics.

\begin{center}
\scalebox{0.8}{
\begin{tabular}{cccc}
  \href{http://www.findstat.org/St000022}{\tt St000022},
\href{http://www.findstat.org/St000215}{\tt St000215},
\href{http://www.findstat.org/St000241}{\tt St000241}, \\[-10pt]
\href{http://www.findstat.org/St000338}{\tt St000338},
\href{http://www.findstat.org/St000461}{\tt St000461},
\href{http://www.findstat.org/St000873}{\tt St000873} \\[-10pt]
  $\E(X_\stat^{(n)}) = 1$ \\[-10pt]
  $\V(X_\stat^{(n)}) = 1$ \\[-5pt]
    \begin{tabular}{rcccccc}
     ~$n$ & $\tilde\kappa^{(n)}_3$ & $\tilde\kappa^{(n)}_4$ & $\tilde\kappa^{(n)}_5$ & $\tilde\kappa^{(n)}_6$ & $\tilde\kappa^{(n)}_7$ & $\tilde\kappa^{(n)}_8$ \\[-10pt]
       5  & 1.00 & 1.00 & 1.00 & 0.000 & -14.0 & -118.\\[-10pt]
       6  & 1.00 & 1.00 & 1.00 & 1.00 & 0.000 & -20.0\\
    \end{tabular}
\end{tabular}
}
\scalebox{0.8}{
\begin{tabular}{cccc}
  \href{http://www.findstat.org/St000029}{\tt St000029},
\href{http://www.findstat.org/St000030}{\tt St000030} \\[-10pt]
  $\E(X_\stat^{(n)}) = \tfrac{1}{6}  (n - 1)  (n + 1)$ \\[-10pt]
  $\V(X_\stat^{(n)}) = \tfrac{1}{90}  (n + 1)  (n^{2} + \tfrac{7}{2})$ \\[-5pt]
    \begin{tabular}{rcccccc}
     ~$n$ & $\tilde\kappa^{(n)}_3$ & $\tilde\kappa^{(n)}_4$ & $\tilde\kappa^{(n)}_5$ & $\tilde\kappa^{(n)}_6$ & $\tilde\kappa^{(n)}_7$ & $\tilde\kappa^{(n)}_8$ \\[-10pt]
       6  & -0.283 & -0.362 & 0.425 & 0.858 & -1.70 & -6.60\\[-10pt]
       7  & -0.244 & -0.339 & 0.344 & 0.685 & -1.61 & -3.33\\[-10pt]
       8  & -0.216 & -0.313 & 0.282 & 0.560 & -1.15 & -2.06\\
    \end{tabular}
\end{tabular}
}

\scalebox{0.8}{
\begin{tabular}{cccc}
  \href{http://www.findstat.org/St000039}{\tt St000039},
\href{http://www.findstat.org/St000223}{\tt St000223},
\href{http://www.findstat.org/St000356}{\tt St000356},
\href{http://www.findstat.org/St000358}{\tt St000358} \\[-10pt]
  $\E(X_\stat^{(n)}) = \tfrac{1}{12}  (n - 2)  (n - 1)$ \\[-10pt]
  $\V(X_\stat^{(n)}) = \tfrac{1}{180}  (n - 2)  (n^{2} + \tfrac{11}{2} n - \tfrac{1}{2})$ \\[-5pt]
    \begin{tabular}{rcccccc}
     ~$n$ & $\tilde\kappa^{(n)}_3$ & $\tilde\kappa^{(n)}_4$ & $\tilde\kappa^{(n)}_5$ & $\tilde\kappa^{(n)}_6$ & $\tilde\kappa^{(n)}_7$ & $\tilde\kappa^{(n)}_8$ \\[-10pt]
       6  & 0.564 & -0.0574 & -0.887 & -1.46 & 0.411 & 10.7\\[-10pt]
       7  & 0.494 & -0.0267 & -0.614 & -1.01 & 0.133 & 6.31\\[-10pt]
       8  & 0.448 & -0.00899 & -0.458 & -0.746 & -0.0523 & 3.56\\
    \end{tabular}
\end{tabular}
}
\scalebox{0.8}{
\begin{tabular}{cccc}
  \href{http://www.findstat.org/St000054}{\tt St000054},
\href{http://www.findstat.org/St000740}{\tt St000740} \\[-10pt]
  $\E(X_\stat^{(n)}) = \tfrac{1}{2}  (n + 1)$ \\[-10pt]
  $\V(X_\stat^{(n)}) = \tfrac{1}{12}  (n - 1)  (n + 1)$ \\[-5pt]
    \begin{tabular}{rcccccc}
     ~$n$ & $\tilde\kappa^{(n)}_3$ & $\tilde\kappa^{(n)}_4$ & $\tilde\kappa^{(n)}_5$ & $\tilde\kappa^{(n)}_6$ & $\tilde\kappa^{(n)}_7$ & $\tilde\kappa^{(n)}_8$ \\[-10pt]
       5  & 0.000 & -1.30 & 0.000 & 7.75 & 0.000 & -102.\\[-10pt]
       6  & 0.000 & -1.27 & 0.000 & 7.46 & 0.000 & -96.7\\[-10pt]
       7  & 0.000 & -1.25 & 0.000 & 7.29 & 0.000 & -93.8\\
    \end{tabular}
\end{tabular}
}

\scalebox{0.8}{
\begin{tabular}{cccc}
  \href{http://www.findstat.org/St000060}{\tt St000060} \\[-10pt]
  $\E(X_\stat^{(n)}) = \tfrac{2}{3}  (n - \tfrac{1}{2})$ \\[-10pt]
  $\V(X_\stat^{(n)}) = \tfrac{1}{18}  (n - 2)  (n + 1)$ \\[-5pt]
    \begin{tabular}{rcccccc}
     ~$n$ & $\tilde\kappa^{(n)}_3$ & $\tilde\kappa^{(n)}_4$ & $\tilde\kappa^{(n)}_5$ & $\tilde\kappa^{(n)}_6$ & $\tilde\kappa^{(n)}_7$ & $\tilde\kappa^{(n)}_8$ \\[-10pt]
       5  & -0.600 & -0.800 & 3.00 & 0.400 & -29.4 & 55.6\\[-10pt]
       6  & -0.588 & -0.729 & 2.79 & 0.102 & -25.9 & 52.7\\[-10pt]
       7  & -0.581 & -0.690 & 2.68 & -0.0439 & -24.2 & 51.0\\
    \end{tabular}
\end{tabular}
}
\scalebox{0.8}{
\begin{tabular}{cccc}
  \href{http://www.findstat.org/St000111}{\tt St000111},
\href{http://www.findstat.org/St000471}{\tt St000471} \\[-10pt]
  $\E(X_\stat^{(n)}) = \tfrac{1}{3}  (n - 1)  (n + 1)$ \\[-10pt]
  $\V(X_\stat^{(n)}) = \tfrac{1}{36}  (n + 2)  (n + 1)^{2}$ \\[-5pt]
    \begin{tabular}{rcccccc}
     ~$n$ & $\tilde\kappa^{(n)}_3$ & $\tilde\kappa^{(n)}_4$ & $\tilde\kappa^{(n)}_5$ & $\tilde\kappa^{(n)}_6$ & $\tilde\kappa^{(n)}_7$ & $\tilde\kappa^{(n)}_8$ \\[-10pt]
       6  & -0.251 & -0.216 & 0.309 & 0.206 & -0.935 & 0.430\\[-10pt]
       7  & -0.235 & -0.193 & 0.258 & 0.165 & -0.718 & -0.180\\[-10pt]
       8  & -0.222 & -0.174 & 0.219 & 0.136 & -0.549 & -0.104\\
    \end{tabular}
\end{tabular}
}

\scalebox{0.8}{
\begin{tabular}{cccc}
  \href{http://www.findstat.org/St000154}{\tt St000154},
\href{http://www.findstat.org/St000472}{\tt St000472} \\[-10pt]
  $\E(X_\stat^{(n)}) = \tfrac{1}{6}  (n - 1)  (n + 1)$ \\[-10pt]
  $\V(X_\stat^{(n)}) = \tfrac{1}{36}  (n - 1)  (n + 1)^{2}$ \\[-5pt]
    \begin{tabular}{rcccccc}
     ~$n$ & $\tilde\kappa^{(n)}_3$ & $\tilde\kappa^{(n)}_4$ & $\tilde\kappa^{(n)}_5$ & $\tilde\kappa^{(n)}_6$ & $\tilde\kappa^{(n)}_7$ & $\tilde\kappa^{(n)}_8$ \\[-10pt]
       6  & 0.323 & -0.228 & -0.461 & 0.165 & 1.70 & 1.35\\[-10pt]
       7  & 0.294 & -0.202 & -0.364 & 0.142 & 1.14 & 0.316\\[-10pt]
       8  & 0.270 & -0.182 & -0.297 & 0.122 & 0.825 & 0.151\\
    \end{tabular}
\end{tabular}
}
\scalebox{0.8}{
\begin{tabular}{cccc}
  \href{http://www.findstat.org/St000213}{\tt St000213},
\href{http://www.findstat.org/St000325}{\tt St000325},
\href{http://www.findstat.org/St000470}{\tt St000470},
\href{http://www.findstat.org/St000702}{\tt St000702} \\[-10pt]
  $\E(X_\stat^{(n)}) = \tfrac{1}{2}  (n + 1)$ \\[-10pt]
  $\V(X_\stat^{(n)}) = \tfrac{1}{12}  (n + 1)$ \\[-5pt]
    \begin{tabular}{rcccccc}
     ~$n$ & $\tilde\kappa^{(n)}_3$ & $\tilde\kappa^{(n)}_4$ & $\tilde\kappa^{(n)}_5$ & $\tilde\kappa^{(n)}_6$ & $\tilde\kappa^{(n)}_7$ & $\tilde\kappa^{(n)}_8$ \\[-10pt]
       5  & 0.000 & -0.200 & 0.000 & 0.000 & 0.000 & 10.8\\[-10pt]
       6  & 0.000 & -0.171 & 0.000 & 0.140 & 0.000 & -2.27\\
    \end{tabular}
\end{tabular}
}

\scalebox{0.8}{
\begin{tabular}{cccc}
  \href{http://www.findstat.org/St000235}{\tt St000235},
\href{http://www.findstat.org/St000673}{\tt St000673} \\[-10pt]
  $\E(X_\stat^{(n)}) = (n - 1)$ \\[-10pt]
  $\V(X_\stat^{(n)}) = 1$ \\[-5pt]
    \begin{tabular}{rcccccc}
     ~$n$ & $\tilde\kappa^{(n)}_3$ & $\tilde\kappa^{(n)}_4$ & $\tilde\kappa^{(n)}_5$ & $\tilde\kappa^{(n)}_6$ & $\tilde\kappa^{(n)}_7$ & $\tilde\kappa^{(n)}_8$ \\[-10pt]
       5  & -1.00 & 1.00 & -1.00 & 0.000 & 14.0 & -118.\\[-10pt]
       6  & -1.00 & 1.00 & -1.00 & 1.00 & 0.000 & -20.0\\
    \end{tabular}
\end{tabular}
}
\scalebox{0.8}{
\begin{tabular}{cccc}
  \href{http://www.findstat.org/St000236}{\tt St000236} \\[-10pt]
  $\E(X_\stat^{(n)}) = 2$ \\[-10pt]
  $\V(X_\stat^{(n)}) = 2  (n - 1)^{-1}  (n - 2)$ \\[-5pt]
    \begin{tabular}{rcccccc}
     ~$n$ & $\tilde\kappa^{(n)}_3$ & $\tilde\kappa^{(n)}_4$ & $\tilde\kappa^{(n)}_5$ & $\tilde\kappa^{(n)}_6$ & $\tilde\kappa^{(n)}_7$ & $\tilde\kappa^{(n)}_8$ \\[-10pt]
       5  & 0.272 & -0.556 & -1.09 & 0.741 & 8.89 & 9.46\\[-10pt]
       6  & 0.395 & -0.266 & -0.865 & -0.713 & 2.49 & 12.5\\
    \end{tabular}
\end{tabular}
}

\scalebox{0.8}{
\begin{tabular}{cccc}
  \href{http://www.findstat.org/St000242}{\tt St000242} \\[-10pt]
  $\E(X_\stat^{(n)}) = (n - 2)$ \\[-10pt]
  $\V(X_\stat^{(n)}) = 2  (n - 1)^{-1}  (n - 2)$ \\[-5pt]
    \begin{tabular}{rcccccc}
     ~$n$ & $\tilde\kappa^{(n)}_3$ & $\tilde\kappa^{(n)}_4$ & $\tilde\kappa^{(n)}_5$ & $\tilde\kappa^{(n)}_6$ & $\tilde\kappa^{(n)}_7$ & $\tilde\kappa^{(n)}_8$ \\[-10pt]
       5  & -0.272 & -0.556 & 1.09 & 0.741 & -8.89 & 9.46\\[-10pt]
       6  & -0.395 & -0.266 & 0.865 & -0.713 & -2.49 & 12.5\\
    \end{tabular}
\end{tabular}
}
\scalebox{0.8}{
\begin{tabular}{cccc}
  \href{http://www.findstat.org/St000246}{\tt St000246},
\href{http://www.findstat.org/St000304}{\tt St000304},
\href{http://www.findstat.org/St000692}{\tt St000692},
\href{http://www.findstat.org/St000868}{\tt St000868} \\[-10pt]
  $\E(X_\stat^{(n)}) = \tfrac{1}{4}  (n - 1)  n$ \\[-10pt]
  $\V(X_\stat^{(n)}) = \tfrac{1}{36}  (n - 1)  n  (n + \tfrac{5}{2})$ \\[-5pt]
    \begin{tabular}{rcccccc}
     ~$n$ & $\tilde\kappa^{(n)}_3$ & $\tilde\kappa^{(n)}_4$ & $\tilde\kappa^{(n)}_5$ & $\tilde\kappa^{(n)}_6$ & $\tilde\kappa^{(n)}_7$ & $\tilde\kappa^{(n)}_8$ \\[-10pt]
       5  & 0.000 & -0.468 & 0.000 & 1.13 & 0.000 & -6.40\\[-10pt]
       6  & 0.000 & -0.377 & 0.000 & 0.750 & 0.000 & -3.55\\[-10pt]
       7  & 0.000 & -0.317 & 0.000 & 0.539 & 0.000 & -2.18\\
    \end{tabular}
\end{tabular}
}

\scalebox{0.8}{
\begin{tabular}{cccc}
  \href{http://www.findstat.org/St000279}{\tt St000279} \\[-10pt]
  $\E(X_\stat^{(n)}) = 1$ \\[-10pt]
  $\V(X_\stat^{(n)}) = \tfrac{1}{6}  (n - 1)  (n + 4)$ \\[-5pt]
    \begin{tabular}{rcccccc}
     ~$n$ & $\tilde\kappa^{(n)}_3$ & $\tilde\kappa^{(n)}_4$ & $\tilde\kappa^{(n)}_5$ & $\tilde\kappa^{(n)}_6$ & $\tilde\kappa^{(n)}_7$ & $\tilde\kappa^{(n)}_8$ \\[-10pt]
       5  & 3.20 & 12.3 & 49.6 & 165. & 18.0 & -7640.\\[-10pt]
       6  & 4.41 & 27.4 & 211. & 1790. & 15200. & 113000.\\[-10pt]
       7  & 5.79 & 53.5 & 679. & 10300. & 171000. & 2.91e6\\
    \end{tabular}
\end{tabular}
}
\scalebox{0.8}{
\begin{tabular}{cccc}
  \href{http://www.findstat.org/St000355}{\tt St000355},
\href{http://www.findstat.org/St000359}{\tt St000359} \\[-10pt]
  $\E(X_\stat^{(n)}) = \tfrac{1}{12}  (n - 2)  (n - 1)$ \\[-10pt]
  $\V(X_\stat^{(n)}) = \tfrac{1}{60}  (n - 2)  (n^{2} - \tfrac{1}{3} n + \tfrac{1}{3})$ \\[-5pt]
    \begin{tabular}{rcccccc}
     ~$n$ & $\tilde\kappa^{(n)}_3$ & $\tilde\kappa^{(n)}_4$ & $\tilde\kappa^{(n)}_5$ & $\tilde\kappa^{(n)}_6$ & $\tilde\kappa^{(n)}_7$ & $\tilde\kappa^{(n)}_8$ \\[-10pt]
       6  & 0.761 & -0.0432 & -1.96 & -3.81 & 4.92 & 54.8\\[-10pt]
       7  & 0.704 & 0.0300 & -1.41 & -2.96 & 1.52 & 31.3\\[-10pt]
       8  & 0.663 & 0.0757 & -1.07 & -2.38 & 0.0499 & 19.1\\
    \end{tabular}
\end{tabular}
}

\scalebox{0.8}{
\begin{tabular}{cccc}
  \href{http://www.findstat.org/St000357}{\tt St000357},
\href{http://www.findstat.org/St000360}{\tt St000360} \\[-10pt]
  $\E(X_\stat^{(n)}) = \tfrac{1}{12}  (n - 2)  (n - 1)$ \\[-10pt]
  $\V(X_\stat^{(n)}) = \tfrac{1}{60}  (n - 2)  (n^{2} + \tfrac{13}{6} n - \tfrac{43}{6})$ \\[-5pt]
    \begin{tabular}{rcccccc}
     ~$n$ & $\tilde\kappa^{(n)}_3$ & $\tilde\kappa^{(n)}_4$ & $\tilde\kappa^{(n)}_5$ & $\tilde\kappa^{(n)}_6$ & $\tilde\kappa^{(n)}_7$ & $\tilde\kappa^{(n)}_8$ \\[-10pt]
       5  & 1.46 & 2.27 & 2.30 & -8.03 & -80.0 & -388.\\[-10pt]
       6  & 1.28 & 1.84 & 2.19 & -1.74 & -33.2 & -195.\\[-10pt]
       7  & 1.14 & 1.49 & 1.74 & -0.229 & -15.6 & -92.6\\[-10pt]
       8  & 1.03 & 1.24 & 1.35 & 0.0677 & -8.65 & -48.7\\
    \end{tabular}
\end{tabular}
}
\scalebox{0.8}{
\begin{tabular}{cccc}
  \href{http://www.findstat.org/St000462}{\tt St000462},
\href{http://www.findstat.org/St000463}{\tt St000463},
\href{http://www.findstat.org/St000866}{\tt St000866},
\href{http://www.findstat.org/St000961}{\tt St000961} \\[-10pt]
  $\E(X_\stat^{(n)}) = \tfrac{1}{4}  (n - 2)  (n - 1)$ \\[-10pt]
  $\V(X_\stat^{(n)}) = \tfrac{1}{36}  (n - 2)  (n + \tfrac{1}{2})  (n + 3)$ \\[-5pt]
    \begin{tabular}{rcccccc}
     ~$n$ & $\tilde\kappa^{(n)}_3$ & $\tilde\kappa^{(n)}_4$ & $\tilde\kappa^{(n)}_5$ & $\tilde\kappa^{(n)}_6$ & $\tilde\kappa^{(n)}_7$ & $\tilde\kappa^{(n)}_8$ \\[-10pt]
       5  & 0.142 & -0.674 & -0.142 & 2.53 & 0.222 & -22.3\\[-10pt]
       6  & 0.0754 & -0.482 & -0.0309 & 1.11 & -0.329 & -5.89\\[-10pt]
       7  & 0.0446 & -0.376 & -0.00857 & 0.690 & -0.0656 & -2.78\\[-10pt]
       8  & 0.0284 & -0.311 & -0.00278 & 0.483 & -0.0204 & -1.74\\
    \end{tabular}
\end{tabular}
}

\scalebox{0.8}{
\begin{tabular}{cccc}
  \href{http://www.findstat.org/St000500}{\tt St000500} \\[-10pt]
  $\E(X_\stat^{(n)}) = n$ \\[-10pt]
  $\V(X_\stat^{(n)}) = (n - 1)  (n + 2)$ \\[-5pt]
    \begin{tabular}{rcccccc}
     ~$n$ & $\tilde\kappa^{(n)}_3$ & $\tilde\kappa^{(n)}_4$ & $\tilde\kappa^{(n)}_5$ & $\tilde\kappa^{(n)}_6$ & $\tilde\kappa^{(n)}_7$ & $\tilde\kappa^{(n)}_8$ \\[-10pt]
       5  & 1.05 & 0.847 & -0.436 & -5.40 & -20.6 & -46.2\\[-10pt]
       6  & 1.08 & 1.01 & 0.287 & -2.76 & -13.4 & -45.6\\
    \end{tabular}
\end{tabular}
}
\scalebox{0.8}{
\begin{tabular}{cccc}
  \href{http://www.findstat.org/St000619}{\tt St000619} \\[-10pt]
  $\E(X_\stat^{(n)}) = \tfrac{1}{2}  n$ \\[-10pt]
  $\V(X_\stat^{(n)}) = \tfrac{1}{12}  n$ \\[-5pt]
    \begin{tabular}{rcccccc}
     ~$n$ & $\tilde\kappa^{(n)}_3$ & $\tilde\kappa^{(n)}_4$ & $\tilde\kappa^{(n)}_5$ & $\tilde\kappa^{(n)}_6$ & $\tilde\kappa^{(n)}_7$ & $\tilde\kappa^{(n)}_8$ \\[-10pt]
       5  & 0.000 & -0.240 & 0.000 & 1.92 & 0.000 & -39.4\\[-10pt]
       6  & 0.000 & -0.200 & 0.000 & 0.000 & 0.000 & 10.8\\[-10pt]
       7  & 0.000 & -0.171 & 0.000 & 0.140 & 0.000 & -2.27\\
    \end{tabular}
\end{tabular}
}

\scalebox{0.8}{
\begin{tabular}{cccc}
  \href{http://www.findstat.org/St000724}{\tt St000724} \\[-10pt]
  $\E(X_\stat^{(n)}) = \tfrac{2}{3}  (n + 1)$ \\[-10pt]
  $\V(X_\stat^{(n)}) = \tfrac{1}{18}  (n - 2)  (n + 1)$ \\[-5pt]
    \begin{tabular}{rcccccc}
     ~$n$ & $\tilde\kappa^{(n)}_3$ & $\tilde\kappa^{(n)}_4$ & $\tilde\kappa^{(n)}_5$ & $\tilde\kappa^{(n)}_6$ & $\tilde\kappa^{(n)}_7$ & $\tilde\kappa^{(n)}_8$ \\[-10pt]
       5  & -0.600 & -0.800 & 3.00 & 0.400 & -29.4 & 55.6\\[-10pt]
       6  & -0.588 & -0.729 & 2.79 & 0.102 & -25.9 & 52.7\\[-10pt]
       7  & -0.581 & -0.690 & 2.68 & -0.0439 & -24.2 & 51.0\\
    \end{tabular}
\end{tabular}
}
\scalebox{0.8}{
\begin{tabular}{cccc}
  \href{http://www.findstat.org/St000756}{\tt St000756} \\[-10pt]
  $\E(X_\stat^{(n)}) = n$ \\[-10pt]
  $\V(X_\stat^{(n)}) = \tfrac{1}{2}  (n - 1)  n$ \\[-5pt]
    \begin{tabular}{rcccccc}
     ~$n$ & $\tilde\kappa^{(n)}_3$ & $\tilde\kappa^{(n)}_4$ & $\tilde\kappa^{(n)}_5$ & $\tilde\kappa^{(n)}_6$ & $\tilde\kappa^{(n)}_7$ & $\tilde\kappa^{(n)}_8$ \\[-10pt]
       5  & 0.632 & -0.100 & -1.35 & -2.07 & 2.91 & 24.0\\[-10pt]
       6  & 0.689 & 0.0889 & -1.04 & -2.25 & -0.0675 & 15.8\\[-10pt]
       7  & 0.727 & 0.222 & -0.790 & -2.18 & -1.63 & 9.59\\
    \end{tabular}
\end{tabular}
}

\scalebox{0.8}{
\begin{tabular}{cccc}
  \href{http://www.findstat.org/St000809}{\tt St000809} \\[-10pt]
  $\E(X_\stat^{(n)}) = \tfrac{1}{12}  (n - 1)  (n + 4)$ \\[-10pt]
  $\V(X_\stat^{(n)}) = \tfrac{1}{180}  (n^{3} + \tfrac{7}{2} n^{2} + \tfrac{7}{2} n + 16)$ \\[-5pt]
    \begin{tabular}{rcccccc}
     ~$n$ & $\tilde\kappa^{(n)}_3$ & $\tilde\kappa^{(n)}_4$ & $\tilde\kappa^{(n)}_5$ & $\tilde\kappa^{(n)}_6$ & $\tilde\kappa^{(n)}_7$ & $\tilde\kappa^{(n)}_8$ \\[-10pt]
       6  & 0.194 & -0.153 & -0.338 & -0.203 & 0.840 & 2.65\\[-10pt]
       7  & 0.232 & -0.102 & -0.308 & -0.277 & 0.396 & 1.92\\[-10pt]
       8  & 0.253 & -0.0672 & -0.273 & -0.292 & 0.202 & 1.59\\
    \end{tabular}
\end{tabular}
}
\scalebox{0.8}{
\begin{tabular}{cccc}
  \href{http://www.findstat.org/St000825}{\tt St000825} \\[-10pt]
  $\E(X_\stat^{(n)}) = \tfrac{1}{2}  (n - 1)  n$ \\[-10pt]
  $\V(X_\stat^{(n)}) = \tfrac{1}{18}  (n - 1)  n  (n + 7)$ \\[-5pt]
    \begin{tabular}{rcccccc}
     ~$n$ & $\tilde\kappa^{(n)}_3$ & $\tilde\kappa^{(n)}_4$ & $\tilde\kappa^{(n)}_5$ & $\tilde\kappa^{(n)}_6$ & $\tilde\kappa^{(n)}_7$ & $\tilde\kappa^{(n)}_8$ \\[-10pt]
       5  & 0.000 & -0.143 & 0.000 & -0.109 & 0.000 & 0.0396\\[-10pt]
       6  & 0.000 & -0.101 & 0.000 & -0.0304 & 0.000 & -0.160\\[-10pt]
       7  & 0.000 & -0.0800 & 0.000 & 0.00121 & 0.000 & -0.125\\
    \end{tabular}
\end{tabular}
}

\scalebox{0.8}{
\begin{tabular}{cccc}
  \href{http://www.findstat.org/St000830}{\tt St000830} \\[-10pt]
  $\E(X_\stat^{(n)}) = \tfrac{1}{3}  (n - 1)  (n + 1)$ \\[-10pt]
  $\V(X_\stat^{(n)}) = \tfrac{2}{45}  (n + 1)  (n^{2} + \tfrac{7}{2})$ \\[-5pt]
    \begin{tabular}{rcccccc}
     ~$n$ & $\tilde\kappa^{(n)}_3$ & $\tilde\kappa^{(n)}_4$ & $\tilde\kappa^{(n)}_5$ & $\tilde\kappa^{(n)}_6$ & $\tilde\kappa^{(n)}_7$ & $\tilde\kappa^{(n)}_8$ \\[-10pt]
       6  & -0.283 & -0.362 & 0.425 & 0.858 & -1.70 & -6.60\\[-10pt]
       7  & -0.244 & -0.339 & 0.344 & 0.685 & -1.61 & -3.33\\[-10pt]
       8  & -0.216 & -0.313 & 0.282 & 0.560 & -1.15 & -2.06\\
    \end{tabular}
\end{tabular}
}
\scalebox{0.8}{
\begin{tabular}{cccc}
  \href{http://www.findstat.org/St000962}{\tt St000962} \\[-10pt]
  $\E(X_\stat^{(n)}) = \tfrac{1}{4}  (n - 4)  (n - 3)$ \\[-10pt]
  $\V(X_\stat^{(n)}) = \tfrac{5}{8}  (n - 4)  (n - 3)$ \\[-5pt]
    \begin{tabular}{rcccccc}
     ~$n$ & $\tilde\kappa^{(n)}_3$ & $\tilde\kappa^{(n)}_4$ & $\tilde\kappa^{(n)}_5$ & $\tilde\kappa^{(n)}_6$ & $\tilde\kappa^{(n)}_7$ & $\tilde\kappa^{(n)}_8$ \\[-10pt]
       6  & 0.998 & -0.174 & -3.91 & -6.08 & 28.4 & 174.\\[-10pt]
       7  & 0.548 & -0.589 & -1.70 & 1.43 & 13.9 & -0.347\\[-10pt]
       8  & 0.311 & -0.590 & -0.679 & 1.87 & 3.98 & -14.6\\
    \end{tabular}
\end{tabular}
}

\scalebox{0.8}{
\begin{tabular}{cccc}
  \href{http://www.findstat.org/St001084}{\tt St001084} \\[-10pt]
  $\E(X_\stat^{(n)}) = \tfrac{1}{6}  (n - 2)$ \\[-10pt]
  $\V(X_\stat^{(n)}) = \tfrac{1}{18}  (n - 2)  (n - \tfrac{1}{2})$ \\[-5pt]
    \begin{tabular}{rcccccc}
     ~$n$ & $\tilde\kappa^{(n)}_3$ & $\tilde\kappa^{(n)}_4$ & $\tilde\kappa^{(n)}_5$ & $\tilde\kappa^{(n)}_6$ & $\tilde\kappa^{(n)}_7$ & $\tilde\kappa^{(n)}_8$ \\[-10pt]
       6  & 1.57 & 1.38 & -4.43 & -29.2 & -59.7 & 307.\\[-10pt]
       7  & 1.54 & 1.30 & -4.36 & -27.7 & -53.6 & 300.\\
    \end{tabular}
\end{tabular}
}

\end{center}

\bibliographystyle{plain}
\bibliography{math}

\end{document}